\documentclass{amsart}

\usepackage{graphicx}
\usepackage{amssymb}
\usepackage{enumerate}
\usepackage{mathtools}
\usepackage{soul,color}
\newtheorem{theorem}{Theorem}[section]
\newtheorem{lemma}[theorem]{Lemma}
\newtheorem{coro}{Corollary}[section]
\theoremstyle{definition}
\newtheorem{definition}[theorem]{Definition}

\newtheorem{prop}{Proposition}[section]

\usepackage{xcolor}

\theoremstyle{plain}
\newtheorem{maintheorem}{Theorem}

\theoremstyle{remark}
\newtheorem{remark}{\bf{Remark}}[section]

\numberwithin{equation}{section}

\begin{document}

\title{Non-uniform hyperbolicity of maps on $\mathbb{T}^2$}

\author{Sebasti\'an A. Ramirez}
\address{Facultad de Matem\'aticas, Pontificia Universidad Cat\'olica de Chile, Santiago, Chile}
\email{sramired@uc.cl}
\thanks{The first author was supported by Proyecto FONDECYT Postdoctorado 3240422, ANID, Chile.}

\author{Kendry J. Vivas}
\address{Departamento de Matem\'aticas, Universidad Católica del Norte, Antofagasta, Chile.}
\email{kendry.vivas01@ucn.cl}

\subjclass[2010]{Primary 37A05, 37D25.}


\keywords{Lyapunov exponents, Non-uniform hyperbolicity, Stable ergodicity}

\begin{abstract}
In this paper we prove that the homotopy class of  non-homothety linear endomorphisms on $\mathbb{T}^2$ with determinant greater than 2 contains a $C^1$ open set of non-uniformly hyperbolic endomorphisms. Furthermore, we prove that the homotopy class of non-hyperbolic elements (having either $1$ or $-1$ as an eigenvalue) whose degree is large enough contains non-uniformly hyperbolic endomorphisms that are also $C^2$ stably ergodic. These results provide partial answers to certain questions posed in \cite{ACS}. 
\end{abstract}

\maketitle

\section{Introduction and statements of the results}\label{statements}

Lyapunov exponents have proved to be in the last years a powerful tool to studying chaos in dynamical systems. 
These quantities give the exponential rate of expansion or contraction of vectors along the orbits of a system. This theory turns in an active research field thanks to the works of Furstenberg, Kesten, Kingman, Ledrappier, Oseledets and others between the sixties and eighties, and it is present, for instance, in the study of random walks on groups and the Shr\"odinger operators. In this work, we study the particular interaction of Lyapunov exponents with smooth dynamics. 

Throughout this paper, we consider smooth conservative endomorphisms $f:\mathbb{T}^2\to\mathbb{T}^2$, i.e., non-invertible mappings preserving the Lebesgue (Haar) measure $\mu$. For a smooth map $f$ and a pair $(\overline{x},u)\in T\mathbb{T}^2$, the \textit{Lyapunov exponent of $f$ at $(\overline{x},u)$} is given by 
\begin{displaymath}
\chi(\overline{x},u)\coloneqq\limsup_{n\to\infty}\frac{1}{n}\log\Vert Df^n(\overline{x})u\Vert. 
\end{displaymath}

According to Oseledet's theorem \cite{O} there is a full area subset $M$ of $\mathbb{T}^2$ such that the above limit exists for every $\overline{x}\in M$ and every $u\neq 0$. Moreover, there is a measurable bundle $E^-$ and measurable functions $\chi^-$ and $\chi^+$ defined on $M$ such that 
\begin{displaymath}
\chi(\overline{x},u)=\lim_{n\to\infty}\frac{1}{n}\log(min(Df^n(\overline{x})))=\chi^-(\overline{x}),\quad\forall u\in E^-(\overline{x}),
\end{displaymath}
where $min(A)=\inf\{\Vert A v\Vert : \Vert v \Vert=1\}$ is the conorm of a linear map $A$, and
\begin{displaymath}
\chi(\overline{x},u)=\lim_{n\to\infty}\frac{1}{n}\log\Vert Df^n(\overline{x})\Vert=\chi^+(\overline{x}),\quad\forall u\in \mathbb{R}^2\setminus E^-(\overline{x}). 
\end{displaymath}
It is easy to check that $\chi^+(\overline{x})\geq \chi^-(\overline{x})$ almost everywhere and  
\begin{displaymath}
    \int [\chi^+(\overline{x})+\chi^-(\overline{x})]d\mu(\overline{x})=\int \log \vert\det Df(\overline{x})\vert d\mu(\overline{x})>0, 
\end{displaymath}
so that $\chi^+(\overline{x})$ is positive almost everywhere. In this way, we recall the definition of non-uniform hyperbolic system $f$ given in \cite{ACS}: 
\begin{definition}
The map $f$ is non-uniformly hyperbolic (NUH for short) if $\chi^-(\overline{x})<0<\chi^+(\overline{x})$ almost everywhere. 
\end{definition}

The notion of NUH was introduced by Y. Pesin in \cite{P} in order to generalize the classical hyperbolic theory. Trivial examples of NUH maps are the Anosov diffeomorphisms. The first example of a NUH map  that is not Anosov was exhibited by A. Katok in \cite{K}. It consist of a ``slowdown'' of a linear Anosov map near the origin by a local perturbation. Moreover, it is shown that any surface supports diffeomorphisms satisfying this property. Later, this result was generalized in \cite{DP}, showing that, in fact, any closed manifold supports this kind of map. More recently, Berger and Carrasco \cite{BC} constructed a volume-preserving partially hyperbolic diffeomorphism in $\mathbb T^4$ whose two-dimensional central direction has the NUH property and has no dominated splitting. Furthermore, this construction is $C^2$-robust among volume-preserving diffeomorphisms.  

The Bochi-Mañe theorem \cite{BM} shows that NUH area-preserving diffeomorphisms on surfaces are fragile in the following sense: unless a diffeomorphism is Anosov, any conservative diffeomorphism $f$ can be approximated in the topology $C^1$ by another diffeomorphism with zero exponents.
Moreover, there exists a generic set of linear cocycles $A:M\to SL(2,\mathbb R)$, which exhibit two distinct behaviors: Either exhibit uniform hyperbolicity or have both Lyapunov exponents equal to zero. 
However, M. Viana and J. Yang in \cite{VY} proved that the above result does not hold in the non-invertible case by exhibiting a $C^0$-open set of $SL(2,\mathbb R)$-cocycles whose Lyapunov exponents are bounded far away from zero. 

Now, it is well known that any map $f:\mathbb{T}^2\to\mathbb{T}^2$ is homotopic to a linear endomorphism induced by an integer matrix $E$ that we denote by the same letter. So, in what follows we consider homotopy classes associated to linear endomorphisms $E$ such that $\vert\det E\vert\geq 2$. These classes of maps consist of non-invertible local diffeomorphisms. 

In \cite{ACS}, M. Andersson, P. Carrasco and R. Saghin showed that there exists a $C^1$-open set of NUH maps that intersects every homotopy class of linear endomorphisms $E$ which are not homotheties whose degree is bigger than 5. In order to enunciate this result, we consider the set $End_{\mu}^1(\mathbb{T}^2)$ of $C^1$ local diffeomorphisms of $\mathbb{T}^2$ preserving the Lebesgue measure $\mu$. For $f\in End_{\mu}^1(\mathbb{T}^2)$, define the number
\begin{displaymath}
    C_{\chi}(f):=\sup_{n\in\mathbb{N}}\inf_{(\overline{x},u)\in T^1\mathbb{T}^2}\frac{1}{n}I(\overline{x},u;f^n),
\end{displaymath}
where
\begin{displaymath}
    I(\overline{x},u;f^n)=\sum_{\overline{y}\in f^{-n}(\overline{x})}\frac{\log\Vert (Df^n(\overline{y}))^{-1}u\Vert }{\vert \det (Df^n(\overline{y}))\vert}.
\end{displaymath}

\begin{remark}
A similar expression for $I(\overline{x},u;f)$ was considered in \cite{M} to define a weak*-convergent sequence to an invariant measure $\mu^-$ for a non-invertible smooth map $f$ called \textit{inverse SRB measure}. Furthermore, this measure is supported on a hyperbolic repellor and it satisfies a Pesin-type formula involving the negative Lyapunov exponents of $\mu^-$.  
\end{remark}

Let consider 
\begin{displaymath}
    \mathcal{U}:=\lbrace f\in End_{\mu}^1(\mathbb{T}^2) : C_{\chi}(f)>0\rbrace.  
\end{displaymath}
By definition of $C_{\chi}(f)$, we see that $\mathcal{U}$ is $C^1$-open.
For $f\in End_{\mu}^1(\mathbb{T}^2)$, denote by $[f]$ the class of $C^1$ smoothly homotopic maps of $f$, i.e., 
\begin{displaymath}
    [f]=\lbrace g:\mathbb{T}^2\to\mathbb{T}^2 : \text{ there is a }C^1\text{ homotopy between }f\text{ and }g\rbrace.
\end{displaymath}
\begin{theorem}[Theorem A in \cite{ACS}]\label{teoACS}
Any $f\in\mathcal{U}$ is non-uniformly hyperbolic. Moreover, if $E=(e_{ij})\in M_{2\times 2}(\mathbb{Z})$ is not a homothety and $\frac{\vert\det E\vert}{gcd (e_{ij})}>4$, then the intersection $[E]\cap \mathcal{U}$ is non-empty, and in fact contains maps that are real analytically homotopic to $E$. 
\end{theorem}

The above result shows that, unlike \cite[Theorem A]{BM}, the rigidity phenomenon is not present in the context of endomorphisms. Moreover, the NUH maps constructed in \cite{ACS} have no dominated splitting in a robust way.

In light of Theorem \ref{teoACS}, the following question was posed in \cite{ACS}:

\textbf{Question 1: }Is i true that $\mathcal{U}$ intersects all the homotopy classes of endomorphisms on $\mathbb{T}^2$? 

Recently, V. Janeiro in \cite{J} extends Theorem \ref{teoACS} to homotheties whose degree is bigger than $5^2$ and some small degree cases. 
In this paper, we will prove the following result: 
\begin{maintheorem}\label{mainteo}
If $E\in M_{2\times 2}(\mathbb{Z})$ is a non-homothety with $\vert\det E\vert>2$, then the intersection $ [E]\cap \mathcal{U}$ is non-empty, and in fact contains maps that are real analytically homotopic to $E$. 
\end{maintheorem}

In this way, by combining Theorem A with \cite[Theorem B]{J} we obtain a partial answer to Question 1 for the non-homothety case. 

\begin{coro}\label{corolario1}
The set $\mathcal{U}$ intersects every homotopy class of endomorphisms in $\mathbb{T}^2$ associated to non-homotheties $E$ with $\vert\det E\vert> 2$. 
\end{coro}

It should be noted that Theorem A is equivalent to \cite[Theorem B]{J}, but our result includes some cases that the author of the mentioned reference did not consider. Besides, by taking a look on the proof of \cite[Theorem B]{ACS} we observe that this result can be extended to the cases studied here. 

Now, following an analogous definition  in \cite{Oba2018}, we say that an endomorphism $f$ is $C^2$-\textit{stably ergodic} for $\mu$ if there is a neighborhood $\mathcal{U'}$ of $f$ in $End_{\mu}^2(\mathbb{T}^2)$ such that every $g \in \mathcal{U'}$ is also ergodic. It should be noted that in this definition we use a $C^2$-neighborhood instead of a $C^1$-neighborhood since the argument depends on the control of the bounds of the $C^2$-norms in a neighborhood of $f$.  

The theory of stably ergodicity was initiated in the pioneering work of M. Grayson, C. Pugh, and M. Shub \cite{MPS}, and exhibit the time-one map of the geodesic flow on the unit tangent bundle of a surface with negative constant curvature as a first example. This theory has been widely studied in the context of diffeomorphisms. Indeed, in \cite{PS} was showed that $C^2$ volume-preserving partially hyperbolic diffeomorphisms satisfying certain conditions (stable dynamical coherence and stable accessibility) are stably ergodic. In \cite{BW}, a characterization of stable ergodicity and the denseness of this property were obtained for skew products. Examples of stably ergodic diffeomorphisms which are not partially hyperbolic were given in \cite{T} and \cite{NORH} respectively. In the context of endomorphsisms, the following result was proved in \cite{ACS}: 
\begin{theorem}\label{teoCACS}
For any linear endomorphisms $E$ as in Theorem \ref{teoACS}, if $\pm1$ is not an eigenvalue of $E$, then $[E]\cap\mathcal{U}$ contains $C^1$ stably ergodic endomorphisms, i.e., there is a neighborhood $\mathcal{U'}$ of $f$ in $End_{\mu}^1(\mathbb{T}^2)$ so that every $g \in \mathcal{U'}$ of class $C^2$ is also ergodic.
\end{theorem}

So, according to above result the following question was posed in \cite{ACS}:

\textbf{Question 2:} Are there stably ergodic non-uniformly hyperbolic endomorphisms in every homotopy class of endomorphisms on $\mathbb{T}^2$?

It should be noted that the argument of Theorem \ref{teoCACS} relies on the classical Hopf argument and a result due to M. Andersson regarding the transitivity of area-preserving endomorphisms on $\mathbb{T}^2$. It should be noted that this result and the non-domination property of these maps show that Corollary 5.2 of \cite{AM} does not hold for non invertible mappings. In this paper, we provide a partial answer to Question 2 by proving $C^2$-stably ergodicity for linear maps $E$ having either $1$ or $-1$ as an eigenvalue and its degree is large enough.    

\begin{maintheorem}\label{main2}
Let $E$ be a linear endomorphism as in Theorem A having $\pm1$ as one of its eigenvalues, and let $\delta_0\in(0,1)$ such that
\begin{equation}\label{estimatet}
\left\lfloor\frac{1+7\delta_0}{28\delta_0}\right\rfloor>20.   
\end{equation}
If $\vert\det E\vert\geq m_0$, where $m_0\in\mathbb{N}$ satisfies
\begin{equation}\label{n}
    \frac{\lfloor\frac{m_0-1}{2}\rfloor-1}{\lfloor\frac{m_0-1}{2}\rfloor+1}>(1-\delta_0),
\end{equation}
then $[E]\cap\mathcal{U}$ contains non-uniformly hyperbolic endomorphisms that are $C^2$-stably ergodic.
\end{maintheorem}

In this context, it is straightforward to see that if the Lebesgue measure $\mu$ is ergodic for $g$, then the restriction of $g$ to the support of the measure $\mu$ is transitive (i.e., it has at least one dense orbit). Since $\mu$ is fully supported, we have the following direct consequence.

\begin{coro}
In the homotopy classes considered in Theorem B, there exists a set of non-uniformly hyperbolic endomorphisms that are $C^2$-robustly transitive.
\end{coro}

It is worth mentioning that the corollary above carries the same flavor as Proposition 3.7 in \cite{HG}, in the sense that we establish the existence of robustly transitive endomorphisms even in homotopy classes whose linear representative is not hyperbolic.


\section{Proof of Theorem A}

\subsection{Shears and its induced dynamics}

Before to prove Theorem \ref{mainteo}, some previous results are needed.

Let us consider 
\begin{displaymath}
    \mathcal{H}_{NH}(\mathbb{T}^2)=\lbrace E\in M_{2\times 2}(\mathbb{Z}) : \det E\neq 0,E\neq k\cdot Id, k\in\mathbb{R} \rbrace. 
\end{displaymath}
For $E=(e_{ij})\in\mathcal{H}_{NH}(\mathbb{T}^2)$, define $\tau_1=\tau_1(E)$ and $\tau_2=\tau_2(E)$ as 
\begin{displaymath}
\tau_1=gcd(e_{ij})\quad\text{and}\quad\tau_2=\frac{\vert\det E\vert}{\tau_1}. 
\end{displaymath}
By definition both numbers are integers, $\tau_1$ divides $\tau_2$ and $d=\tau_1\cdot\tau_2$, where $d$ denotes the degree of $E$. These numbers are called \textit{elementary divisors of $E$}. 
\begin{remark}
For matrices $E$ as in Theorem \ref{teoACS}, the pair
$(\tau_1,\tau_2)$ satisfies $\tau_2\geq 5$. On the other hand, \cite[Theorem  B]{J} covers all these pairs and it includes the pairs $(\tau_1,\tau_2)=(3,3), (4,4)$. 
\end{remark}

By \cite[Proposition 2.3]{ACS}, up to a linear change of coordinates, we can assume that
\begin{displaymath}
E=
\left( \begin{array}{cc}
e_{11} & e_{12} \\
e_{21}  & e_{22} 
\end{array} \right),
\end{displaymath}
where $e_{ij} \in\mathbb{Z}$ depends on $(\tau_1,\tau_2)$ and $e_{11}e_{22}-e_{12}e_{21}=d$. Besides, it is possible to make the linear change of coordinates in such a way that the vector $e_2=(0,1)$ is not an eigenvector of $E$, i.e., $e_{12} \neq 0$. Moreover, for every $\overline{x}\in\mathbb{T}^2$ one has 
\begin{displaymath}
    E^{-1}(\overline{x})=\overline{y}+\left\lbrace \left(\frac{i}{\tau_2},\frac{j}{\tau_1}\right) 
 : \begin{tabular}{l} $i=0,\ldots, \tau_2-1$, \\
$j=0,\ldots, \tau_1-1$ \end{tabular}\right\rbrace, 
\end{displaymath}
where $y$ is the unique point in $\mathbb{R}^2$ satisfying $Ey=x$. This implies that the preimages are separated by a distance of $\frac{1}{\tau_2}$ along the horizontal coordinate.

The main strategy in \cite{ACS} to find non-uniformly hyperbolic maps homotopic to $E$ is to consider a one-parameter family of area-preserving diffeomorphisms $h_t$, $t\in\mathbb{R}$, called \textit{shears}. These shears deform the linear endomorphism $E$ in a such way as to obtain an element $f_t\in\mathcal{U}$, which implies that $f_t$ has a negative Lyapunov exponent. For this, the authors studied the dynamics of $(Dh_t(\overline{y}))^{-1}$, where $\overline{y}$ is a pre-image of an element $\overline{x}\in\mathbb{T}^2$. A key step in their argument is to exploit the fact that $E$ has large degree (and hence, so does $f_ t$). Indeed, this property and the definition of $h_t$ imply that there are a region $\mathcal{G}$ containing the preimage (except at most one point) of every element of $\mathbb{T}^2$, and a cone field $\Delta_{\alpha}$ (which is invariant on $\mathcal{G}$) with strong expansion on $\mathcal{G}$ under $(Df_t(\cdot))^{-1}$. In this way, the average of $\log\Vert(Df_t^n(\overline{y}))^{-1}u\Vert$ over the entire backward $n$-orbit of any point $(\overline{x},u)\in T^1\mathbb{T}^2$ is uniformly positive as $n$ goes to infinity. In \cite{J}, this reasoning was extended to homotheties with a degree of at least $5^2$ and some small-degree cases by considering suitable families of shears. 

In the specific case of non-homotheties, the above procedure cannot be applied to low-degree matrices because the estimates they provided for $I(x,y;f^n)$
are only strictly positive when $\tau_2\geq 5$. To overcome this difficulty and to adapt the argument given in \cite{ACS} in our context, it is necessary to consider higher-order pre-images of any $\overline{x}\in\mathbb{T}^2$. The main challenge at this point is both to study the distribution of these points on $\mathbb{T}^2$ and to study the track of the horizontal and vertical vectors along each pre-orbit obtained in this process.

As a first step  in proving our result, we must consider the higher order pre-images of a point $x\in\mathbb{T}^2$ and determine the best way in which these points must be distributed. First, we will consider a  very simple analytic function $s:\mathbb{T}^1\to\mathbb{R}$ that captures the essential dynamics of the map $f_t$ derived from it, and we set $f_t=E\circ h_t$, $t\in\mathbb{R}$, where 
\begin{displaymath}
h_t(\overline{x})=(x,y+ts(x)),\quad \forall \overline{x}=(x,y)\in\mathbb{T}^2. 
\end{displaymath}
Note that the definitions of both $h_t$ and $f_t$ are the same to that given in \cite{ACS}. In particular, $h_t$ is inspired in the classical standard map \cite{C}, where the map $s(x)$ plays the role of $\sin(2\pi x)$. When $s$ is a smooth map, $h_t$ is a area-preserving diffeomorphism and $f_t$ is $C^1$ homotopic to $E$. In coordinates, $f_t(\overline{x})$ and $f_t^{-1}(\overline{x})$, for $\overline{x}=(x,y)\in\mathbb{T}^2$, can be written as
\begin{equation}\label{defif}
    f_t(\overline{x})=(e_{11}x+e_{12} (y+ts(x)),e_{21} x+e_{22} (y+ts(x)))
\end{equation}
 and 
\begin{equation}\label{defifin}
    f_t^{-1}(\overline{x})=\left\lbrace \left(\psi_1(x,y,i),\psi_2(x,y,j)-ts\left(\psi_1(x,y,i)\right)\right) :\begin{tabular}{l} $i=0,\ldots, \tau_2-1$, \\
$j=0,\ldots, \tau_1-1$ \end{tabular} \right\rbrace, 
\end{equation}
where 
\begin{displaymath}
    \psi_1(x,y,i)=\frac{1}{d}(e_{22}x-e_{12}y)+\frac{i}{\tau_2}\quad\text{and}\quad\psi_2(x,y,j)=\frac{1}{d}(e_{11}y-e_{21}x)+\frac{j}{\tau_1}.
\end{displaymath}

Take the partition $0<\frac{1}{\tau_2+1}<\ldots<\frac{\tau_2}{\tau_2+1}$ of $\mathbb{T}^1$. Note that $\frac{j}{\tau_2+1}+\frac{i}{\tau_2}\neq \frac{k}{\tau_2+1}$ for every $j,k=0,\ldots, \tau_2$ and $i=1,\ldots,\tau_2-1$. We will construct a piecewise linear function within each of the sub-intervals defined by the partition, alternating between positive and negative slopes. More precisely, define $\hat{s}:\mathbb{T}^1\to\mathbb{R}$ as follows:
\begin{enumerate}[(a)]
\item Let $\hat{s}$ be continuous and  piece-wise linear in $J_j=\left[\frac{j}{\tau_2+1},\frac{j+1}{\tau_2+1}\right)$ with slopes $a_j\in\mathbb{R}$, such that $a_j<0$ if $j$ is odd and $a_j>0$ otherwise, and $\vert a_{j+1}\vert>2\vert a_j\vert$ for every $j=0,\ldots, \tau_2$. In this way, $x_j=\frac{j}{\tau_2+1}$, $j=0\ldots, \tau_2$ are the critical points of $\hat{s}$. 
\item $s_j=\hat{s}\left(\frac{j}{\tau_2+1}\right), \lim_{x\to 1}\hat{s}(x)\in\mathbb{Q}\setminus\mathbb{Z}$, for $j=0,\ldots \tau_2$, and $\hat{s}(0)-\lim_{x\to 1}\hat{s}(x)\in\mathbb{Z}$.
\end{enumerate}

Let $\hat{f}_t=E\circ \hat{h}_t$, where $\hat{h}_t(\overline{x})=(x,y+t\hat{s}(x))$, for any $ \overline{x}=(x,y)\in\mathbb{T}^2$. Note that $\hat{f}_t$ satisfies the equations \eqref{defif} and \eqref{defifin}. Define the \textit{critical set} $\mathcal{C}_{\tau_2}$ as
\begin{displaymath}
    \mathcal{C}_{\tau_2}=\bigcup \ell_{j},\quad \ell_j=\lbrace x_j\rbrace\times \mathbb{T}^1,\quad j=0,\ldots, \tau_2. 
\end{displaymath} 

Already defined the function $\hat{s}$, let $I_j=\left[\frac{j}{\tau_2+1}-\delta,\frac{j}{\tau_2+1}+\delta\right]$, $j=0,\ldots, \tau_2$, where $\delta>0$ satisfies $\left(I_j+\frac{i}{\tau_2}\right)\cap I_k=\emptyset$ for every $j,k=0,\ldots, \tau_2$ with $j\neq k$. In a similar way to \cite{ACS}, we define the \textit{critical region} and the \textit{good region} as 
   \begin{displaymath}
        \mathcal{C}=\left(\bigcup_{j=0}^{\tau_2} I_j\right)\times \mathbb{T}^1,
    \end{displaymath}
and $\mathcal{G}=\mathcal{G}^-\cup\mathcal{G}^+$, where 
    \begin{displaymath}
        \mathcal{G}^-=\left(\bigcup_{j\text{ is odd}}J'_j\right)\times \mathbb{T}^1\quad\text{and}\quad\mathcal{G}^+=\left(\bigcup_{j\text{ is even}}J'_j\right)\times\mathbb{T}^1,  
    \end{displaymath}
respectively, with $J_j'=J_j\setminus (I_j\cup I_{j+1})$, for $j=0,\ldots, \tau_2$. 

Then, we make the map $\hat{s}$ to be analytical on $\bigcup_{j}I_j$ with zero derivative at $\frac{j}{\tau_2+1}$, and $s=\hat{s}$ on $\mathbb{T}^1\setminus\mathcal{C}$, to obtain an analytical map $s:\mathbb{T}^1\to\mathbb{R}$ and a smooth local diffeomorphism $f_t$ which is $C^1$ homotopic to $E$. 

\begin{remark}\label{rm2.2}
The following observations can be derived from the construction above, both of which will be useful later in Lemma 2.3.
    \begin{itemize}
        \item The set $\mathcal{G}$ consists of $\tau_2+1$ strips $A_j=J'_j\times\mathbb{T}^1$, each of width $\ell < \frac{1}{\tau_2}$ for $j=0,\ldots,\tau_2$, and each of these strips is invariant under $h_t$. In this way, by definition of $f$, each of these strips can contain at most  $\tau_1$ points of $f^{-1}(\overline{x})$ for every $\overline{x}\in\mathbb{T}^2$.  
        \item When $\tau_2$ is even, there will be pairs of strips with the same slope sign. Nevertheless, the specific choice of slopes within the different strips $A_j$ allows us to effectively adapt the argument given in \cite{ACS}.
    \end{itemize}
\end{remark}

Next lemma shows that the analytic map $f_t$ obtained from  above construction, for certain values of $t$, has a good distribution of its higher order pre-images. 

\begin{lemma}\label{L2}
Let $E\in \mathcal{H}_{NH}(\mathbb{T}^2)$ with elementary divisors  $(\tau_1,\tau_2)$, $\tau_2\geq 3$. Then, there exists an analytic function $s:\mathbb{T}^1\to\mathbb{R}$ such that the map $f_t=E\circ h_t$ is $C^1$-homotopic to $E$ and satisfies the following: For every point $\overline{x}\in\mathbb{T}^2$ the set $f_t^{-1}(\overline{x})$ has $d$ elements, of which there is at least $\tau_1\lfloor\frac{\tau_2-1}{2}\rfloor$ of them inside each one of $\mathcal{G}^{-}$ and $\mathcal{G}^+$, and at most $\tau_1$ of them inside $\mathcal{C}$. Furthermore, at least one pre-image $\overline{y}\in\mathcal{G}$ satisfies $d(\overline{y},\mathcal{C})>\frac{1}{8(d+1)}$. In addition,  there are infinitely many and arbitrarily large numbers $t>0$ such that \begin{equation}\label{threeconse}
    \mathcal{C}\cap f_t^{-1}(\mathcal{C})\cap f_t^{-2}(\mathcal{C})=\emptyset. 
\end{equation}    
\end{lemma} 

\begin{proof}
The first property is a direct consequence of the construction of $f_t$ and Remark \ref{rm2.2}.

On the other hand, let $\delta>0$ such that 
\begin{displaymath}
    \frac{1}{\tau_2+1}-2\delta>\frac{1}{2(\tau_2+1)}.
\end{displaymath}
Notice, by the choice of $\delta$, that every connected component $J$ of the good region has size bigger than $\frac{1}{2(\tau_2+1)}$, so that a band of size $\frac{1}{4(\tau_2+1)}$ in the middle of some of these components $J$ contains one pre-image. Thus, the distance of this point to the boundary is bigger than $\frac{1}{2}\left(\frac{1}{2(\tau_2+1)}-\frac{1}{4(\tau_2+1)}\right)=\frac{1}{8(\tau_2+1)}>\frac{1}{8(d+1)}$.   

For the last part, let $\hat{f}_t$ given in the construction of $f_t$, and let $\overline{x}=(x,y)\in\mathcal{C}_{\tau_2}\cap \hat{f}_t^{-1}(\mathcal{C}_{\tau_2})$. Then, $x=\frac{j}{\tau_2+1}$, for some $j=0,\ldots, \tau_2$, and by (\ref{defif}), item (b), and definition of $\overline{x}$ one has by a straightforward computation that
\begin{displaymath}
    \overline{x}=\left(\frac{j_0}{\tau_2+1},\frac{j_1-e_{11}j_0}{e_{12} (\tau_2+1)}-ts_{j_0}\right),\quad j_0,j_1=0,\ldots, \tau_2
\end{displaymath}

Now, notice by \eqref{defifin} that
\begin{displaymath}
   \pi_1(\hat{f}^{-1}(\overline{x}))=\frac{j}{d(\tau_2+1)}+\frac{i}{\tau_2}+\frac{e_{12}s_j}{d}t,\quad j=0, \ldots, d(\tau_2+1)-1,\quad i=0,\ldots,\tau_2-1,
\end{displaymath}
where $\pi_1:\mathbb{T}^2\to\mathbb{T}^1$ is the projection on the first coordinate. Hence, if $t>0$ satisfies 
\begin{displaymath}
    t\neq \frac{1}{s_je_{12}}\left(n_1+\frac{n_2}{\tau_2(\tau_2+1)}\right), \quad n_1, n_2\in\mathbb{Z}, j=0,\ldots,\tau_2,
\end{displaymath}
we have that $\hat{f}_t^{-1}(\overline{x})\subset\mathcal{G}$. Moreover, it is possible to choose $t$ in a such way that $\hat{f}_t^{-1}(\overline{x})$ is away from $\mathcal{C}_{\tau_2}$. This shows that
\begin{equation}\label{bbg}
    \mathcal{C}_{\tau_2}\cap \hat{f}_t^{-1}(\mathcal{C}_{\tau_2})\cap \hat{f}_t^{-2}(\mathcal{C}_{\tau_2})=\emptyset. 
\end{equation}
In other words, the above relation says that for any point $\overline{x}\in\mathbb{T}^2$, there are not three consecutive pre-images in the critical set $\mathcal{C}_{\tau_2}$.
Therefore, the result is obtained by taking $\delta>0$ such that \eqref{bbg} holds for $f_t$ and $\mathcal{C}$. 
\end{proof}

\subsection{Dynamics of $Df_t$.}

Now, we are interested in studying the dynamics of the linear map $Df_{t}$. For this, we follow \cite{ACS}. Since the expression of $C_{\chi}(f)$ is independent of the choice of the norm on the tangent bundle of $\mathbb{T}^2$, the authors in the mentioned reference chose the maximum norm by the sake of simplicity in the computations.

Recall that if $\alpha>0$, the \textit{horizontal cone} $\Delta_{\alpha}^h$ is defined as 
\begin{displaymath}
    \Delta_{\alpha}^h=\lbrace (u_1,u_2)\in\mathbb{R}^2 : \vert u_2\vert\leq \alpha\vert u_1\vert\rbrace,
\end{displaymath}
while the \textit{vertical cone} $\Delta_{\alpha}^v$ is given by $\Delta_{\alpha}^v=\mathbb{R}^2\setminus \Delta_{\alpha}^h$. It should be note that if  $E\in\mathcal{H}_{NH}(\mathbb{T}^2)$, there is $\alpha>1$ such that
\begin{equation}\label{verticaltohorizontal}
E^{-1}\Delta_{\alpha}^v\subset\overline{E^{-1}\Delta_{\alpha}^v}\subset \textit{Int}(\Delta_{\alpha}^h)\subset \Delta_{\alpha}^h. 
\end{equation}

In \cite{ACS} was proved that the following properties hold for $(Df_t)^{-1}$: 
\begin{enumerate}[(NH1)]
    \item If $\overline{y}\in\mathcal{G}$, then $\Delta_{\alpha}^v$ is strictly invariant for $(Df_{t}(\overline{y}))^{-1}$. 
    \item If $u\in \Delta_{\alpha}^v$ is a unit vector, then 
    \begin{displaymath}
        \Vert(Df_t(\overline{y}))^{-1}u\Vert\geq m((Df_t(\overline{y}))^{-1})>
        \begin{cases}
\displaystyle\frac{e_v(a-\frac{\alpha}{t})}{\alpha}t  &  \text{if  }\overline{y}\in\mathcal{G}\\
\displaystyle\frac{e_v}{\alpha}   &  \text{if }\overline{y}\in\mathcal{C},
\end{cases}
    \end{displaymath}
    where $e_v=\inf_{u\in\Delta_{\alpha}^v, \Vert u\Vert=1}\Vert E^{-1}u\Vert$. 
    \item If $u\in\Delta_{\alpha}^h$, and $E^{-1}u=(u_1,u_2)$, let $*(u)$ be the sign of $-\frac{u_1}{u_2}$ (with the convention that $0$ and $\pm\infty$ have both $+$ and $-$ sign). Then, for all $\overline{y}\in\mathcal{G}^{*(u)}$ we have $(Df_t(\overline{y}))^{-1}u\in\Delta_{\alpha}^v$. 
    \item If $u\in \Delta_{\alpha}^h$ is a unit vector, then 
    \begin{displaymath}
        \Vert(Df_t(\overline{y}))^{-1}u\Vert\geq m((Df_t(\overline{y}))^{-1})>
        \begin{cases}
e_h  &  \text{if  }\overline{y}\in\mathcal{G}^{*(u)}\\
\displaystyle\frac{e_h}{(b+\frac{1}{t})}t^{-1}   &  \text{if }\overline{y}\notin\mathcal{G}^{*(u)},
\end{cases}
    \end{displaymath}
    where $e_h=\inf_{v\in\Delta_{\alpha}^h, \Vert u\Vert=1}\Vert E^{-1}u\Vert$,
\end{enumerate}
for every $t>\frac{2\alpha}{a}$, where $0<a<b$ are the lower and upper bounds, respectively, of the derivative of $s$ on $\mathbb{T}^1\setminus\left(\bigcup_{j=0}^{\tau_2}I_j\right)$. For the maps given here one has $a=a_0$ and $b=\vert a_{\tau_2}\vert$.

Another useful tool for the proof of Theorem \ref{mainteo} is the next lemma, whose proof is analogous to that showed for Lemma 2.2 in \cite{ACS} by taking $g=f^k$:  
\begin{lemma}\label{I}
For $f\in End_{\mu}^1(\mathbb{T}^2)$ and any $n,k\in\mathbb{N}$ we have 
\begin{displaymath}
    I(\overline{x},u;f^{kn})=\sum_{i=0}^{n-1}\sum_{\overline{y}\in f^{-ki}(\overline{x})}\frac{1}{\det (Df^{ki}(\overline{y}))}I(\overline{y},F^{-ki}(\overline{y})u;f^k), 
\end{displaymath}
where $F^{-ki}(\overline{y})u=\displaystyle\frac{(Df^{ki}(\overline{y}))^{-1}u}{\Vert(Df^{ki}(\overline{y}))^{-1}u\Vert}$. In other words, $\frac{1}{nk}I(\overline{x},u;f^{nk})$ is a convex combination of other $I(\overline{y},w;f^k)$. 
\end{lemma} 

\subsection{Key Lemmas.}

The next lemmas will be useful for proving Theorem A. Recall that $d=\tau_1\cdot\tau_2$. 

\begin{lemma}\label{keylemma1}
Let $E\in\mathcal{H}_{NH}(\mathbb{T}^2)$ whose elementary divisors are $(\tau_1,\tau_2)$, $\tau_2\geq 3$, and let $f=f_t$ as in Lemma \ref{L2}. Then, for every $\overline{x}\in\mathbb{T}^2$ and every unit vector $u\in T^1\mathbb{T}^2$ we have the following:
\begin{itemize}
    \item There are at least 
    \begin{displaymath}
     \tau_1^3\left[(\tau_2-1)^3+(\tau_2-1)\left(3\left\lfloor\frac{\tau_2-1}{2}\right\rfloor+1\right)-\left\lfloor\frac{\tau_2-1}{2}\right\rfloor^2\right]   
    \end{displaymath}
    vectors in $\Delta_{\alpha}^v$ for  $(Df^3)^{-1}$ if $u\in\Delta_{\alpha}^v$.
    \item There are at least
    \begin{displaymath}
            \tau_1^3\left[(\tau_2-1)^3-\left(\tau_2-1-\left\lfloor\frac{\tau_2-1}{2}\right\rfloor\right)^3+(\tau_2-1)\left(3\left\lfloor\frac{\tau_2-1}{2}\right\rfloor+1\right)-\left\lfloor\frac{\tau_2-1}{2}\right\rfloor^2\right]
    \end{displaymath}
    vectors in $\Delta_{\alpha}^v$ for  $(Df^3)^{-1}$ if $u\in\Delta_{\alpha}^h$.
\end{itemize}
\end{lemma}

\begin{proof}
Let $u\in\Delta_{\alpha}^v$. Assume that  $f^{-1}(\overline{x})$ has one critical point. In this case, by Lemma \ref{L2} and property (NH1) we have that there are $\tau_1(\tau_2-1)$ vectors in $\Delta_{\alpha}^v$ and $\tau_1$ vectors in $\Delta_{\alpha}^h$ under the action of $(Df(\overline{y}))^{-1}$, where $\overline{y}\in f^{-1}(\overline{x})$.

Now, for each of the $\tau_1(\tau_2-1)$ vertical vectors from the previous step, we again apply Lemma \ref{L2} and (NH1) to obtain $(\tau_1(\tau_2-1))^2$ vectors in $\Delta_{\alpha}^v$ and $\tau_1^2(\tau_2-1)$ vectors in $\Delta_{\alpha}^h$. By applying Lemma \ref{L2} and property (NH1) once more, we obtain $(\tau_1(\tau_2-1))^3$ vectors in $\Delta_{\alpha}^v$ associated to the $(\tau_1(\tau_2-1))^2$ vertical vectors. Simultaneously, for the $\tau_1(\tau_2-1)$ horizontal vectors we obtain $\tau^3(\tau_2-1)$  vectors in $\Delta_{\alpha}^v$.  
 
For the remaining $\tau_1$ horizontal vectors $v$, we apply Lemma \ref{L2} and property (NH3) to obtain $\tau_1^2\lfloor\frac{\tau_2-1}{2}\rfloor$ vectors  in $\Delta_{\alpha}^v$ and $\tau_1^2\left(\tau_2-1-\lfloor\frac{\tau_2-1}{2}\rfloor\right)$ vectors in $\Delta_{\alpha}^h$, which are associated to points in $\mathcal{G}^{*(v)}$ and $\mathcal{G}^{-*(v)}$ respectively, and $\tau_1^2$ vectors  in $\Delta_{\alpha}^h$. 

On one hand, for the $\tau_1^2\lfloor\frac{\tau_2-1}{2}\rfloor$ vertical vectors we obtain $\tau_1^3(\tau_2-1)\lfloor\frac{\tau_2-1}{2}\rfloor$ vectors in $\Delta_{\alpha}^v$ by Lemma \ref{L2} and property (NH1). For the $\tau_1^2\left(\tau_2-1-\lfloor\frac{\tau_2-1}{2}\rfloor\right)$ horizontal vectors we obtain $\tau_1^3\lfloor\frac{\tau_2-1}{2}\rfloor\left(\tau_2-1-\lfloor\frac{\tau_2-1}{2}\rfloor\right)$ vectors in $\Delta_{\alpha}^v$ by property (NH3). 

On the other hand, notice that each one of the $\tau_1^2$ horizontal vectors $w$ are associated to points $\overline{z}\in \mathcal{C}\cap f^{-1}(\mathcal{C})$. Hence by \eqref{threeconse} one has $f^{-1}(\overline{z})\subset\mathcal{G}$. In this case, if $w\in T_z \mathbb T^2$ is a unit vector and $w'=E^{-1}w\in\Delta_{\alpha}^h$, then by (NH1) we get $\tau_1\tau_2$ vectors in $\Delta_{\alpha}^v$ for $(Df)^{-1}$, while if $w'=\pm(w_0,1)\in\Delta_{\alpha}^v$ and $\overline{q}\in\mathcal{G}$, we see that $(Dh_t(\overline{q}))^{-1}w'\in\Delta_{\alpha}^h$ if and only if $\frac{1}{\vert a_i\vert t+\alpha}\leq\vert w_0\vert\leq\frac{1}{\vert a_i\vert t-\alpha}$ for some $i=0,\ldots, \tau_2$. So, since $s=\hat{s}$ on $\mathcal{G}$, and the slopes of $\hat{s}$ on $\mathcal{G}$ satisfy $\vert a_{j+1}\vert>2\vert a_j\vert$ for every $j=0,\ldots, \tau_2$ (item (a) of the construction of $\hat{s}$), we have by Remark \ref{rm2.2} that $\vert w_0\vert\notin\left[\frac{1}{\vert a_j\vert t+\alpha},\frac{1}{\vert a_j\vert t-\alpha}\right]$ for every $j\neq i$. This shows that $Dh_t(\overline{q})w'\in\Delta_{\alpha}^v$ for every $\overline{q}\in\left(\mathcal{G}\setminus(J'_i\times\mathbb{T}^1)\right)$. Thus, there are $\tau_1(\tau_2-1)$ vectors in $\Delta_{\alpha}^v$ for $(Df)^{-1}$. Therefore, we obtain at least $\tau_1^3(\tau_2-1)$ vectors in $\Delta_{\alpha}^v$.

Therefore, we get the first item by adding the above quantities. In a similar way the second item is obtained. 
\end{proof}

For every $n\in\mathbb{N}$ any nonzero tangent vector $(\overline{x},u)\in T^1\mathbb{T}^2$ define 
\begin{eqnarray*}
Df^{-3n}(\overline{x},u)&=&\lbrace(\overline{y},w)\in T\mathbb{T}^2 : f^{3n}(\overline{y})=\overline{x}, Df^{3n}(\overline y)w=u\rbrace, \\
\mathcal{G}_n&=&\lbrace (\overline z,w)\in Df^{-3n}(\overline x,v) : w\in\Delta_{\alpha}^v\rbrace, \\ 
\mathcal{B}_n&=&Df^{-3n}(\overline{x},u)\setminus\mathcal{G}_n, \\ 
g_n&=&\#\mathcal{G}_n, \\
b_n&=&\#\mathcal{B}_n=d^{3n}-g_n. 
\end{eqnarray*}

For later reference, we define the numbers given in Lemma \ref{keylemma1} as $v_1(d)=\tau_1^3v_1(\tau_2)$ and $v_2(d)=\tau_1^3v_2(\tau_2)$.
 Notice that  
\begin{displaymath}
\frac{g_{n+1}}{d^{n+1}}\geq \left(\frac{1}{\tau_2^3}\left(v_1(\tau_2)-v_2(\tau_2)\right)\right)\cdot\frac{g_n}{d^n}+\frac{1}{\tau_2^3}\cdot v_2(\tau_2),\quad \forall n\in\mathbb{N}.
\end{displaymath}
Let $a_{n+1}=\frac{g_{n+1}}{d^{n+1}}$, $n\in\mathbb{N}$, $c=\frac{1}{\tau_2^3}\left(v_1(\tau_2)-v_2(\tau_2)\right)$ and $e=\frac{1}{\tau_2^3}\cdot v_2(\tau_2)$. It should be noted that
\begin{displaymath}
    0<c=\frac{1}{\tau_2^3}\left(v_1(\tau_2)-v_2(\tau_2)\right)=\left(\frac{\tau_2-1-\left\lfloor\frac{\tau_2-1}{2}\right\rfloor}{\tau_2}\right)^3<1. 
\end{displaymath}
Hence, by induction on $n$, 
\begin{displaymath}
    a_{n}\geq \frac{e}{1-c}\cdot(1-c^n)=\frac{v_2(\tau_2)}{\tau_2^3-\left(v_1(\tau_2)-v_2(\tau_2)\right)}\cdot(1-c^n),\quad\forall n\in\mathbb{N}. 
\end{displaymath} 
Thus, as in \cite{ACS}, define
\begin{equation}\label{proportion}
    p(\tau_2)=\liminf a_n=\frac{v_2(\tau_2)}{\tau_2^3-\left(v_1(\tau_2)-v_2(\tau_2)\right)}\in (0,1), 
\end{equation}
which represents the ``asymptotic ratio" of vertical vectors among the pre-images of $(\overline{x},u)$ under $Df$. 
\begin{lemma}
    Let $E\in\mathcal{H}_{NH}(\mathbb{T}^2)$ whose elementary divisors are $(\tau_1,\tau_2)$, $\tau_2\geq 3$, and let $f=f_t$ as in Lemma \ref{L2}. Then, 
    \begin{displaymath}
        p(\tau_2)>\frac{1}{2}. 
    \end{displaymath}
\end{lemma}

\begin{proof} 
Notice that by Lemma \ref{keylemma1} we have 
\begin{eqnarray*}
\frac{v_2(\tau_2)}{\tau_2^3-\left(v_1(\tau_2)-v_2(\tau_2)\right)}&=&1+\frac{v_1(\tau_2)-\tau_2^3}{\tau_2^3-\left(v_1(\tau_2)-v_2(\tau_2)\right)}\\
 &=&1+\frac{N(\tau_2)}{D(\tau_2)},
\end{eqnarray*}
where 
\begin{displaymath}
 N(\tau_2)=-3\tau_2^2+3\tau_2-1+(\tau_2-1)\left(3\left\lfloor\frac{\tau_2-1}{2}\right\rfloor+1\right)-\left\lfloor\frac{\tau_2-1}{2}\right\rfloor^2   
\end{displaymath}
and 
\begin{displaymath}
D(\tau_2)=\tau_2^3-\left(\tau_2-1-\left\lfloor\frac{\tau_2-1}{2}\right\rfloor\right)^3.  
\end{displaymath}

Now, since $\tau_2\geq 3$ and $\tau_2-1>\left\lfloor\frac{\tau_2-1}{2}\right\rfloor$ we have 
\begin{eqnarray*}
0>N(\tau_2)&>&-3\tau_2^2+3\tau_2-1+(\tau_2-1)^2\\
&=&-2\tau_2^2+\tau_2>-2\tau_2^2
\end{eqnarray*}
and 
\begin{eqnarray*}
D(\tau_2)&\geq &\left(1+\left\lfloor\frac{\tau_2-1}{2}\right\rfloor\right)\left(3\tau_2^2-3\tau_2\left(\frac{\tau_2+1}{2}\right)+\left(1+\left\lfloor\frac{\tau_2-1}{2}\right\rfloor\right)^2\right)\\
&=&\left(1+\left\lfloor\frac{\tau_2-1}{2}\right\rfloor\right)\left(\frac{1}{2}(3\tau_2^2-\tau_2)+\left\lfloor\frac{\tau_2-1}{2}\right\rfloor^2\right) \\
&\geq &\left(1+\left\lfloor\frac{\tau_2-1}{2}\right\rfloor\right)\left(2\tau_2^2+\left\lfloor\frac{\tau_2-1}{2}\right\rfloor^2\right)\\
&\geq &4(\tau_2^2+2). 
\end{eqnarray*}

Therefore, 
\begin{displaymath}
    p(\tau_2)=1+\frac{N(\tau_2)}{D(\tau_2)}>1-\frac{2\tau_2^2}{4(\tau_2^2+2)}\geq 1-\frac{1}{2}=\frac{1}{2}. 
\end{displaymath} 
This proves the result. 
\end{proof}

\begin{remark}\label{imporemark}
Notice that for $\tau_2\geq 5$ one has 
\begin{displaymath}
    \frac{2\tau_2^2}{\left(1+\left\lfloor\frac{\tau_2-1}{2}\right\rfloor\right)\left(2\tau_2^2+\left\lfloor\frac{\tau_2-1}{2}\right\rfloor^2\right)}\leq \frac{2\tau_2^2}{6(\tau_2^2+2)}<\frac{1}{3(\tau_2^2+2)},
\end{displaymath}
which shows that $p(\tau_2)>\frac{2}{3}$.  
\end{remark}

Recall that Corollary 3.1 of \cite{ACS} states that
\begin{displaymath}
   I(\overline{x},u;f)\geq \left(1-\frac{1}{\tau_2}\right)\log t+C_1(t),\quad\forall (\overline{x},u)\in T^1\mathbb{T}^2, v\in\Delta_{\alpha}^v
\end{displaymath}
and 
\begin{displaymath}
   I(\overline x,u;f)\geq -\left(1-\frac{1}{\tau_2}\left\lfloor\frac{\tau_2-1}{2}\right\rfloor\right)\log t+C_2(t),\quad\forall (\overline{x},u)\in T^1\mathbb{T}^2, v\in\Delta_{\alpha}^h, 
\end{displaymath}
where $C_k(t)\in\mathbb{R}$, $k=1,2$, depends on $a,b,\tau_1,\tau_2, t$. Denote $c_1=\left(1-\frac{1}{\tau_2}\right)$ and $c_2=-\left(1-\frac{1}{\tau_2}\left\lfloor\frac{\tau_2-1}{2}\right\rfloor\right)$. 

According to above Lemma, denote
\begin{itemize}
    \item $v^v(d)=\tau_1(\tau_2-1)$,
    \item $v^h(d)=\tau_1\left\lfloor\frac{\tau_2-1}{2}\right\rfloor$,
\item $v^v(d^2)=\tau_1^2\left((\tau_2-1)^2+\left\lfloor\frac{\tau_2-1}{2}\right\rfloor\right)$, 
\item $v^h(d^2)=\tau_1^2\left(\left\lfloor\frac{\tau_2-1}{2}\right\rfloor\left(2\tau_2-1-\left\lfloor\frac{\tau_2-1}{2}\right\rfloor\right)+\left\lfloor\frac{\tau_2-1}{2}\right\rfloor\right)$. 
\end{itemize}

\begin{lemma}\label{estimadosvh}
Let $E\in\mathcal{H}_{NH}(\mathbb{T}^2)$ whose elementary divisors are $(\tau_1,\tau_2)$, $\tau_2\geq 3$, and let $f=f_t$ as in Lemma \ref{L2}. Then, for every $\overline{x}\in\mathbb{T}^2$ and every unit vector $u\in T^1\mathbb{T}^2$ we have: For $u\in\Delta_{\alpha}^v$,
\begin{displaymath}
       I(\overline{x},u;f^3)\geq I_1(t)\log t+A(t), 
    \end{displaymath}
where 
\begin{displaymath}
I_1(t)=c_1+(c_1-c_2)\left(\frac{v^v(d)}{d}+\frac{v^v(d^2)}{d^2}\right)+\frac{2d^2-1}{d^2}c_2-\frac{1}{d^2\tau_2}
\end{displaymath}
and $A(t)\in\mathbb{R}$, while for $u\in\Delta_{\alpha}^h$,
\begin{displaymath}
       I(\overline{x},u;f^3)\geq I_2(t)\log t+B(t),
    \end{displaymath}
    where 
    \begin{displaymath}
        I_2(t)=c_2+(c_1-c_2)\left(\frac{v^h(d)}{d}+\frac{v^h(d^2)}{d^2}\right)+\frac{2d^2-1}{d^2}c_2-\frac{1}{d^2\tau_2}
    \end{displaymath}
and $B(t)\in\mathbb{R}$.
\end{lemma}

\begin{proof}
Let $(\overline{x},u)\in T^1\mathbb{T}^2$ a nonzero unit tangent vector, and let $t>\frac{2\alpha}{a}$. By Lemma \ref{I} we have that 
\begin{eqnarray*}
    I(\overline{x},u;f^3)&=&I(\overline{x},u;f)+\frac{1}{d}\sum_{\overline{y}\in f^{-1}(\overline{x})}I(\overline{y},F^{-1}(\overline{y})v;f) \\
    &&+\frac{1}{d^2}\sum_{\overline{z}\in f^{-1}(\overline{y})}I(\overline{z},F^{-1}(\overline{z})w;f).
\end{eqnarray*}

Notice that in the proof of Lemma \ref{keylemma1} we see that for each of the $\tau_1^2$ horizontal vectors $w$ associated to points $\overline{z}\in\mathcal{C}\cap f^{-1}(\mathcal{C})$ one obtain $\tau_1(\tau_2-1)$ vectors in $\Delta_{\alpha}^v$. Hence, by property (NH4) one has 
\begin{displaymath}
    I(\overline{z},w,f)=-\frac{1}{\tau_2}\log t+\left(1-\frac{1}{\tau_2}\right)\log e_h\left(b+\frac{1}{t}\right)^{\tau_1(\tau_2-1)}. 
\end{displaymath}

In this way we have that
\begin{eqnarray*}
I(\overline{x},u;f^3)&\geq& c_1\log t\\
&&+\frac{1}{d}\left(v^v(d)c_1+(d-v^v(d))c_2\right)\log t\\
&&+\frac{1}{d^2}\left(v^v(d^2)c_1+(d^2-1-v^v(d^2))c_2-\frac{1}{\tau_2}\right)\log t+A(t)\\
&=&I_1(t)+A(t). 
\end{eqnarray*}

In a similar way we obtain the second inequality. 
\end{proof}

\begin{prop}
    Let $p=p(\tau_2)\in(0,1)$ as \eqref{proportion}. Then, 
    \begin{equation}\label{firsinqlty}
     J(t)=pI_1(t)+(1-p)I_2(t)>0.   
    \end{equation}
\end{prop}

\begin{proof}
First, note that $pI_1(t)+(1-p)I_2(t)=\left(\sum_{i=1}^3E_i(\tau_2)+3c_2-\frac{1}{d^2}\left(c_2+\frac{1}{\tau_2}\right)\right)\log t$, where 
\begin{displaymath}
E_1(\tau_2)=p(c_1-c_2)(\tau_2)=p\left(2-\frac{1}{\tau_2}\left(\left\lfloor\frac{\tau_2-1}{2}\right\rfloor+1\right)\right),     
\end{displaymath}
\begin{eqnarray*}
E_2(\tau_2)&=&(c_1-c_2)(\tau_2)E_1'(\tau_2)\\
&=&(c_1-c_2)\frac{1}{\tau_2}\left(p\left(\tau_2-1-\left\lfloor\frac{\tau_2-1}{2}\right\rfloor\right)+\left\lfloor\frac{\tau_2-1}{2}\right\rfloor\right)    
\end{eqnarray*}
and 
\begin{eqnarray*}
  E_3(\tau_2)&=&(c_1-c_2)(\tau_2)E_2'(\tau_2)\\
  &=&(c_1-c_2)\frac{1}{\tau_2^2}\left(p\left((\tau_2-1)^2-\left\lfloor\frac{\tau_2-1}{2}\right\rfloor\left(2(\tau_2-1)-\left\lfloor\frac{\tau_2-1}{2}\right\rfloor\right)\right)\right)\\
  &&+(c_1-c_2)\frac{1}{\tau_2^2}\left(\left\lfloor\frac{\tau_2-1}{2}\right\rfloor\left(2\tau_2-1-\left\lfloor\frac{\tau_2-1}{2}\right\rfloor\right)\right).
\end{eqnarray*}

Now, notice that $-\frac{1}{d^2}\left(c_2+\frac{1}{\tau_2}\right)>0$. So, in order to prove the result, we must to prove that 
\begin{equation}\label{2.8}
S(\tau_2)=\sum_{i=1}^3E(\tau_2)+3c_2(\tau_2)>0,\quad \forall\tau_2\geq 3.  
\end{equation}
It is clear that \eqref{2.8} holds for $\tau_2=3$. Assume that \eqref{2.8} holds for $\tau_2=n$. Then, we consider the following cases: 

Case 1: $n$ is odd. 

In this case, we have the following identities:  
\begin{itemize}
    \item $(c_1-c_2)(n+1)=(c_1-c_2)(n)+\frac{1}{n(n+1)}\left(\left\lfloor\frac{n-1}{2}\right\rfloor+1\right)$,
    \item $c_2(n+1)=c_2(n)-\frac{1}{n(n+1)}\left\lfloor\frac{n-1}{2}\right\rfloor$, 
    \item $E_1'(n+1)=E_1'(n)+\frac{p}{n}$, and  
    \item $E_2'(n+1)=E_2'(n)+\frac{1}{n^2}\left(p\left(2n-1-2\left\lfloor\frac{n-1}{2}\right\rfloor\right)+2\left\lfloor\frac{n-1}{2}\right\rfloor\right)$. 
\end{itemize}

On one hand, since $\left\lfloor\frac{n-1}{2}\right\rfloor+1=\left\lfloor\frac{n+1}{2}\right\rfloor\frac{n+1}{2}$, we have 
\begin{displaymath}
    (c_1-c_2)(n)=2-\frac{1}{n}\left(\frac{n+1}{2}\right)\geq \frac{3}{2}-\frac{1}{2n}\geq \frac{4}{3}. 
\end{displaymath}
and 
\begin{displaymath}
    (c_1-c_2)(n)\frac{2}{n^2}\left\lfloor\frac{n-1}{2}\right\rfloor=\frac{8}{3n^2}\left\lfloor\frac{n-1}{2}\right\rfloor\geq \frac{8}{3n(n+1)}\left\lfloor\frac{n-1}{2}\right\rfloor. 
\end{displaymath}

On the other hand, $\frac{1}{n(n+1)}\left(\left\lfloor\frac{n-1}{2}\right\rfloor+1\right)=\frac{1}{2n}$, so that 
\begin{displaymath}
\frac{1}{n(n+1)}\left(\left\lfloor\frac{n-1}{2}\right\rfloor+1\right)E_1'(n+1)\geq \frac{1}{2n^2}\left\lfloor\frac{n-1}{2}\right\rfloor\geq \frac{1}{2n(n+1)}\left\lfloor\frac{n-1}{2}\right\rfloor.    
\end{displaymath}

Hence, by induction hypothesis and the above estimations,
\begin{eqnarray*}
 S(n+1)+3c_2(n+1)&\geq & S(n)+3c_2(n)-\frac{3}{n(n+1)}\left\lfloor\frac{n-1}{2}\right\rfloor\\
 &&+(c_1-c_2)(n)\frac{2}{n^2}\left\lfloor\frac{n-1}{2}\right\rfloor\\
 &&+\frac{1}{n(n+1)}\left(\left\lfloor\frac{n-1}{2}\right\rfloor+1\right)E_1'(n+1)\\
     &>&-\frac{3}{n(n+1)}\left\lfloor\frac{n-1}{2}\right\rfloor+\frac{8}{3n(n+1)}\left\lfloor\frac{n-1}{2}\right\rfloor\\
     &&+\frac{1}{2n(n+1)}\left\lfloor\frac{n-1}{2}\right\rfloor>0.
\end{eqnarray*}

Case 2: $n$ is even. 

In this case, we have the following identities:  
\begin{itemize}
    \item $(c_1-c_2)(n+1)=(c_1-c_2)(n)+\frac{1}{n(n+1)}\left(\left\lfloor\frac{n-1}{2}\right\rfloor+1\right)-\frac{1}{n+1}$,
    \item $c_2(n+1)=c_2(n)+\frac{1}{n(n+1)}\left\lfloor\frac{n-1}{2}\right\rfloor-\frac{1}{n+1}$, 
    \item $E_1'(n+1)=E_1'(n)+\frac{1}{n}$, and  
    \item $E_2'(n+1)=E_2'(n)+\frac{2}{n}$. 
\end{itemize}

On the other hand, since $1+\left\lfloor\frac{n-1}{2}\right\rfloor=\frac{n}{2}$, one has 
\begin{displaymath}
    (c_1-c_2)(n)\geq 2-\frac{1}{n}\left(\left\lfloor\frac{n-1}{2}\right\rfloor\right)=2-\frac{1}{n}\cdot\frac{n}{2}=2-\frac{1}{2}=\frac{3}{2}, 
\end{displaymath}
so that
\begin{displaymath}
\frac{1}{n^2}(c_1-c_2)(n)\left\lfloor\frac{n-1}{2}\right\rfloor\geq \frac{3}{2n(n+1)}\left\lfloor\frac{n-1}{2}\right\rfloor.
\end{displaymath}

Now, by Remark \ref{imporemark} one has $p=p(n+1)\geq \frac{2}{3}$ (because $n\geq 4$), which implies 
\begin{displaymath}
    \frac{1}{n(n+1)}\left(1+\left\lfloor\frac{n-1}{2}\right\rfloor\right)\cdot p>\frac{2}{3n(n+1)}\left\lfloor\frac{n-1}{2}\right\rfloor+\frac{2}{3n(n+1)}. 
\end{displaymath}
Furthermore, 
\begin{eqnarray*}
\frac{1}{n(n+1)}\left(1+\left\lfloor\frac{n-1}{2}\right\rfloor\right)E_1'(n+1)&>&\frac{1}{2(n+1)}\left(\frac{5}{3n}\left\lfloor\frac{n-1}{2}\right\rfloor\right)\\
&=&\frac{5}{6n(n+1)}\left\lfloor\frac{n-1}{2}\right\rfloor. 
\end{eqnarray*}

Besides, since $(n-1)^2\geq (n-1)\geq 2\left\lfloor\frac{n-1}{2}\right\rfloor$, we obtain 
\begin{eqnarray*}
\frac{1}{n(n+1)}\left(1+\left\lfloor\frac{n-1}{2}\right\rfloor\right)E_2'(n+1)&\geq& \frac{1}{2n^2(n+1)}\left\lfloor\frac{n-1}{2}\right\rfloor\\
&&+\frac{1}{2(n+1)}\left(\frac{4}{3n^2}\left\lfloor\frac{n-1}{2}\right\rfloor\right)\\
&=&\frac{2}{3n^2(n+1)}\left\lfloor\frac{n-1}{2}\right\rfloor.  
\end{eqnarray*}

Finally, since $p<1$, 
\begin{displaymath}
    p+E_1'(n+1)+E_2'(n+1)< 3+\frac{1}{n^2}+\frac{1}{n^2}\left\lfloor\frac{n-1}{2}\right\rfloor. 
\end{displaymath}

Therefore, by induction hypothesis and the above estimations, 
\begin{eqnarray*}
    S(n+1)+3c_2(n+1)&\geq& S(n)+3c_2(n)-\frac{3}{n(n+1)}\left\lfloor\frac{n-1}{2}\right\rfloor+\frac{3}{n+1}\\
    &&+\frac{1}{n^2}(c_1-c_2)(n)\left\lfloor\frac{n-1}{2}\right\rfloor\\
    &&+\frac{1}{n(n+1)}\left(1+\left\lfloor\frac{n-1}{2}\right\rfloor\right)\cdot p\\
    &&+\frac{1}{n(n+1)}\left(1+\left\lfloor\frac{n-1}{2}\right\rfloor\right)(E_1'(n+1)+E'_2(n+1))\\
    &&-\frac{1}{n+1}(p+E_1'(n+1)+E_2'(n+1))\\
    &>&-\left(\frac{3}{n(n+1)}+\frac{1}{n^2(n+1)}\right)\left\lfloor\frac{n-1}{2}\right\rfloor\\
    &&+\left(\frac{3}{n(n+1)}+\frac{7}{6n^2(n+1)}\right)\left\lfloor\frac{n-1}{2}\right\rfloor>0.
\end{eqnarray*}

This proves the result. 
\end{proof}

\begin{remark}
    From the proof of above proposition, we see that 
    \begin{displaymath}
        J(t)>-\frac{1}{d^2}\left(c_2+\frac{1}{\tau_2}\right)\log t=\frac{1}{d^2}\left(1-\frac{1}{\tau_2}\left(\left\lfloor\frac{n-1}{2}\right\rfloor+1\right)\right)\log t>0.
    \end{displaymath}
\end{remark}

\subsection{Non-uniform hyperbolicity.}

Now, we are ready to prove Theorem \ref{mainteo}. The proof of this result is a consequence of Key Lemmas and the the argument presented in the proof of \cite[Theorem A]{ACS}.  

\begin{proof}[proof of Theorem \ref{mainteo}]
Define for $i=0,...,n-1$  $$J_i=J_i(\overline{x},u)=\displaystyle\sum_{\overline{y}\in f_{t_0}^{-3i}(\overline{x})}\frac{1}{\det(Df_{t_0}^{3i}(\overline{y}))}I(\overline{y},F_{t_0}^{-3i}(\overline{y})u;f^3).$$

Notice that if $\mathcal{G}_n$ and $\mathcal{B}_n$ as in the previous section, one has
\begin{displaymath}
    J_i=\frac{1}{d^{3i}}\sum_{(y,w)\in\mathcal{G}_i}I(y,w;f^3)+\frac{1}{d^{3i}}\sum_{(y,w)\in\mathcal{G}_i}I(y,w;f^3).
\end{displaymath}

Let $E\in\mathcal{H}_{NH}(\mathbb{T}^2)$ with elementary divisors $(\tau_1,\tau_2)$, $\tau_2\geq 3$. By Proposition 2.1, Lemma 2.5 and Remark 2.4, we have for $t>\frac{2\alpha}{a}$ satisfying the conditions of Lemma \ref{L2} that 
\begin{eqnarray*}
    \lim_{i\to\infty} J_i&\geq & p(\tau_2)I_1(t)+(1-p(\tau_2))I_2(t)+C(t) \\
    &=& J(t)+C(t)\\
    &\geq &\frac{1}{d^2}\left(1-\frac{1}{\tau_2}\left(\left\lfloor\frac{n-1}{2}\right\rfloor+1\right)\right)\log t+C(t),
\end{eqnarray*}
where $C(t)>C\in\mathbb{R}$. Hence, for every non-zero tangent vector $(x,v)$, a large enough $t>0$ and every $i\geq i_0\in\mathbb{N}$, we have $J_i(\overline{x},u)\geq 3c'>0$. Then, by Lemma \ref{I} we have, for $n_0\in\mathbb{N}$ large enough, that 
\begin{displaymath}
    \frac{1}{3n_0}I(\overline{x},u;f_t^{3n_0})=\frac{1}{3n_0}\sum_{i=0}^{n_0-1}J_i(x,v)>\frac{c'}{2}>0,\quad \forall(\overline{x},u)\in T^1\mathbb{T}^2,v\neq 0. 
\end{displaymath}
Therefore, $C_{\chi}(f)>0$. So, by \cite[Proposition 2.2]{ACS}, it follows that $f$ is NUH.
\end{proof}

\section{Proof of Theorem B}

In this section, we will prove Theorem B. The proof of stable ergodicity in \cite{ACS} relies on two main arguments:
\begin{itemize}
    \item Hopf argument: They proved that the stable manifold for almost every point in $\mathbb{T}^2$ has large diameter in order to ensure intersections between stable and unstable manifolds. In particular, this shows that the ergodic components of $\mu$ are open modulo zero sets.
    \item A criterion of transitivity for area-preserving endomorphisms on $\mathbb{T}^2$: If a linear map on $\mathbb{T}^2$ of degree at least two has no real eigenvalues of modulo one, then its whole homotopy class of area-preserving endomorphsisms consists entirely of transitive elements (see \cite[Theorem 2.2]{A}).    
\end{itemize}
The first step of above argument is guaranteed by the NUH property and some properties of the map $f_t$, for $t>0$ large enough. However, for non-hyperbolic matrices $E$, the transitivity criteria can no longer be applicable. So, in our case, to prove stable ergodicity for $\mu$, some additional work is necessary. 

\subsection{Preliminaries}
Recall that for a $C^1$ curve $\gamma:I\subset\mathbb{R}\to\mathbb{T}^2$, whose coordinates are $\gamma(t)=(\gamma_1(t),\gamma_2(t))$, its \textit{length in the maximum} norm is given by 
\begin{displaymath}
\ell_m(\gamma)=\int_{I}\max\lbrace\vert\gamma_1'(t)\vert,\vert\gamma_2'(t)\vert\rbrace dt. 
\end{displaymath}
Note that if $\ell_e(\gamma)$ denotes the euclidean length of $\gamma$, one has
\begin{displaymath}
    \ell_m(\gamma)\leq \ell_e(\gamma)\leq\sqrt{2}\ell_m(\gamma).
\end{displaymath}
From \cite{ACS}, we called $v$-segment to a $C^1$ curve $\gamma$ which is tangent to the vertical cone $\Delta_{\alpha}^v$ and  $\ell=\ell_m(\gamma)=\frac{\alpha}{5e_v}$, where $e_v=\inf_{u\in\Delta_{\alpha}^v, \Vert u\Vert=1}\Vert E^{-1}u\Vert$. In what follows we take $\alpha>1$ such that $\ell>1$. 
\begin{remark}\label{len1}
In \cite{ACS} was observed that the length of the projection on the vertical axis of a $v$-segment $\gamma$ is exactly $\ell$.  In this case we say that $\gamma$ \textit{crosses} $\mathbb{T}^2$ \textit{vertically}.   
\end{remark}
\noindent From above remark, we say that a $C^1$ curve $\gamma'$ \textit{crosses} $\mathbb{T}^2$ \textit{horizontally}if it is tangent to the horizontal cone and its projection on the horizontal axis has size bigger than 1.

Now, for a endomorphism $f:\mathbb{T}^2\to\mathbb{T}^2$, the \textit{natural extension} or \textit{space of pre-orbits of }$f$ is defined as 
\begin{displaymath}
    L_f:=\lbrace \hat{x}=(\overline{x}_0,\overline{x}_1,\overline{x}_2\ldots)\in(\mathbb{T}^2)^{\mathbb{Z}_+} : f(\overline{x}_{i+1})=\overline{x}_i,\forall i\geq 0\rbrace, 
\end{displaymath}
endowed with the product topology. Let $\pi_{ext}:L_f\to\mathbb{T}^2$ be the projection onto the first coordinate, i.e., $\pi_{ext}(\hat{x})=\overline{x}_0$ for every $\hat{x}\in L_f$. Define $\hat{f}:L_f\to L_f$ by 
\begin{displaymath}
    \hat{f}(\hat{x})=(f(\overline{x}_0),\overline{x}_0,\overline{x}_1,\ldots),\quad\hat{x}\in L_f. 
\end{displaymath}
It is easy to check that $\hat{f}$ is a homeomorphism and $\pi_{ext}\circ\hat{f}=f\circ\pi_{ext}$. Besides, it is well known that for any invariant measure $\nu$ for $f$, there is a unique invariant measure $\hat{\nu}$ for $\hat{f}$ such that $(\pi_{ext})_*\hat{\nu}=\nu$. The measure $\hat{\nu}$ is called the \textit{lift} of $\nu$.   

On the other hand, the action of $\pi_{ext}$ on $L_f$ allows us to define a tangent bundle on $L_f$: For $\hat{x}=(\overline{x}_0,\overline{x}_1,\ldots)\in L_f$, let us consider 
\begin{displaymath}
    T_{\hat{x}}L_f=T_{\pi_{ext}(\hat{x})}\mathbb{T}^2=T_{\overline{x}_0}\mathbb{T}^2. 
\end{displaymath}
In the same way, the derivative $Df$ of $f$ lifts to a map $D\hat{f}$ on $T_{\hat{x}}L_f$ in a natural way as $D\hat{f}(\hat{x})=Df(\overline{x}_0)$. Moreover, for every $\hat{x}\in L_f$, 
\begin{displaymath}
D\hat{f}^n(\hat{x})=
\begin{cases}
D\hat{f}(\hat{f}^{n-1}(\hat{x}))\circ \ldots\circ D\hat{f}(\hat{x})=Df^n(\overline{x}_0)  &  \text{if  }n>0\\
Id  &  \text{if  }n=0\\
(D\hat{f}(\hat{f}^{n}(\hat{x})))^{-1}\circ\ldots\circ (D\hat{f}(\hat{f}^{-1}\hat{x}))^{-1}=(Df^n(\overline{x}_n))^{-1}   &  \text{if  }n<0.
\end{cases}
\end{displaymath} 


On the other hand, the relation of the Lyapunov exponents for $f$ and $\hat{f}$ is as follows: 
\begin{eqnarray*}
    \chi_{\hat{f}}(\hat{x},v)&=&\chi_f(\pi_{ext}(\hat{x}),v),\\
    \chi_{\hat{f}}^+(\hat{x})&=&\chi_f^+(\pi_{ext}(\hat{x})),\\
    \chi_{\hat{f}}^-(\hat{x})&=&\chi_f^-(\pi_{ext}(\hat{x})).
\end{eqnarray*}
By \cite[Lemma 2.1]{ACS}, there is a completely invariant set $\hat{\mathcal{R}}\subset L_f$ of full $\hat{\mu}$-measure such that 
\begin{equation}\label{relaxp}
    \chi_{\hat{f}}^+(\hat{x})=-\chi_{\hat{f}^{-1}}^-(\hat{x})=\chi_{f}^+(\overline{x}_0)\quad\text{and}\quad\chi_{\hat{f}}^-(\hat{x})=-\chi_{\hat{f}^{-1}}^+(\hat{x})=\chi_{f}^-(\overline{x}_0),
\end{equation}
where $\overline{x}_0=\pi_{ext}(\hat{x})$, for every $\hat{x}\in\hat{\mathcal{R}}$. Besides, there is a completely invariant set $\mathcal{R}\subset\pi_{ext}(\hat{\mathcal{R}})$ with $\mu(\mathcal{R})=1$ satisfying $\hat{\mu}_x(\hat{\mathcal{R}})=1$ for every $\hat{x}\in\mathcal{R}$, where $\hat{\mu}_x$ is the unique measure on $\pi^{-1}_{ext}(\overline{x})$ satisfying
\begin{displaymath}
\hat{\mu}_x(\lbrace(\xi_0,\xi_1,\ldots)\in\pi^{-1}_{ext}(\overline{x}) : \xi_i=\overline{x}_i\rbrace)=\vert\det Df^i(\overline{x}_i)\vert^{-1}.  
\end{displaymath}
Moreover, for any $\hat{x}\in\hat{\mathcal{R}}$ there are subspaces $E_{\hat{x}}^\pm$ of $\mathbb{R}^2$ such that
\begin{displaymath}
\chi_{\hat{f}}^{\pm}(\hat{x})=\chi_{\hat{f}}(\hat{x},v)=-\chi_{\hat{f}^{-1}}(\hat{x},v)=-\chi^{\mp}_{\hat{f}^{-1}}(\hat{x}),\quad\forall v\in E_{\hat{x}}^{\pm}\setminus\lbrace 0\rbrace. 
\end{displaymath}

Recall the definition of Pesin stable and unstable manifolds given in \cite{ACS}. Let $\hat{x}\in\hat{\mathcal{R}}$, $\pi_{ext}(\hat{x})=\overline{x}$. Unlike the invertible case, the notion of unstable manifold depends on the pre-orbit of a point $\overline{x}\in\mathbb{T}^2$. More precisely, the \textit{local unstable manifold} at $\hat{x}$ is a $C^1$ curve defined by 
\begin{displaymath}
W_{loc}^u(\hat{x})=\left\lbrace \overline{y}\in\mathbb{T}^2
 : \exists !\;\hat{y}\in L_f,\pi_{ext}(\hat{y})=\overline{y},\begin{tabular}{l} $d(x_n,y_n)\leq C_1e^{-n\varepsilon}$ \\
 \;\;\;\;\;\;\;\;\;\;\;\;\;\;\text{and}\\
$d(x_n,y_n)\leq C_2e^{-n\lambda}$ \end{tabular},\forall n\geq0\right\rbrace,
\end{displaymath}
for some constants $\lambda>0$, $0<\varepsilon<\frac{\lambda}{200}$ and $0<C_1\leq 1< C_2$. We denote by $\hat{W}^u_{loc}(\hat{x})$ the lift of $W^u_{loc}(\hat{x})$ to $L_f$ under $\pi_{ext}$. Since the unstable manifold $W^u_{loc}(\hat{x})$ depends on the pre-orbit of $\overline{x}$, there is a bouquet of these manifolds passing through $\overline{x}$. The \textit{unstable manifold} of $f$ at $\hat{x}$ is given by 
\begin{displaymath}
    W^u(\hat{x})=\left\lbrace y_0=\pi_{ext}(\hat{y}) : \limsup_{n\to\infty}\frac{1}{n}\log d(x_n,y_n)<0\right\rbrace, 
\end{displaymath}
and its corresponding lift we denote by $\hat{W}^u(\hat{x})$.
On the other hand, the \textit{local stable manifold} at $\overline{x}$, denoted by $W^s(\overline{x})$, is  a $C^1$ curve defined by  
\begin{displaymath}
W_{loc}^s(\overline{x})=\left\lbrace \overline{y}\in\mathbb{T}^2
 : \begin{tabular}{l} $d(f^n(x),f^n(y))\leq C_1e^{-n\varepsilon}$ \\
 \;\;\;\;\;\;\;\;\;\;\;\;\;\;\text{and}\\
$d(f^n(x),f^n(y))\leq C_2e^{-n\lambda}$ \end{tabular},\forall n\geq0\right\rbrace,
\end{displaymath}
for some constants $\lambda>0$, $0<\varepsilon<\frac{\lambda}{200}$ and $0<C_1\leq 1< C_2$. We denote by $\hat{W}^s_{loc}(\hat{x})$ the lift of $W^s_{loc}(\overline{x})$ to $L_f$ under $\pi_{ext}$. The \textit{stable manifold} of $f$ at $\overline{x}$ is given by 
\begin{displaymath}
    W^s(\overline{x})=\pi_{ext}\left(\bigcup_{n=0}^{\infty}\hat{f}^{-n}(\hat{W}^s_{loc}(\hat{f}^n(\hat{x})))\right). 
\end{displaymath}

According to \cite[Proposition 2.3]{QZ}, there is an increasing countable family $\lbrace\hat{\Lambda}_k\rbrace_{k\geq 0}$ of compact subsets of $\hat{\mathcal{R}}$ such that $\hat{\mu}(\bigcup_{k=0}^{\infty}\hat{\Lambda}_k)=1$, satisfying the following properties: 
\begin{enumerate}
    \item $W^u(\cdot)$ is continuous on $\hat{\Lambda}_k$, and $T_{\overline{x}_0}W^u_{loc}(\hat{x})=E^+(\hat{x})$ for any $\hat{x}\in\hat{\Lambda}_k$, $\pi_{ext}(\hat{x})=\overline{x}_0$, for every $k\geq 0$. Moreover, for every $\hat{x}\in\hat{\Lambda}_k$ there is a sequence of $C^1$ curves $\lbrace W^u(\hat{x},-n)\rbrace_{n\geq0}$ in $\mathbb{T}^2$ such that 
    \begin{itemize}
        \item $W^u(\hat{x},0)=W^u_{loc}(\hat{x})$,
        \item $W^u(\hat{x},-n+1)=f(W^u(\hat{x},n))$, for every $n\geq 1$, 
        \item $W^u(\hat{x})=\bigcup_{n=0}^{\infty}f(W^u_{loc}(\hat{x}),-n)$.
    \end{itemize}
    \item If $\Lambda_k=\pi_{ext}(\hat{\Lambda}_k)$, then $W^s(\cdot)$ is continuous on $\Lambda_k$. Besides, 
    \begin{itemize}
        \item $T_{\overline{x}_0}W^s_{loc}(\overline{x}_0)=E^-(\hat{x})$. Furthermore, $E^-(\cdot)$ is continuous on $\Lambda_k$. 
        \item $f(W^s_{loc}(\overline{x}_0))\subset W^s_{loc}(f(\overline{x}_0))$. 
    \end{itemize}
\end{enumerate}
The sets $\hat{\Lambda}_k$ and $\Lambda_k$ are the so-called \textit{Pesin blocks} for the systems $(\hat{f},\hat{\mu})$ and $(f,\mu)$ respectively. 

Note that the manifolds $W^s$ ($\hat{W}^s$) form an invariant lamination of $\mathbb{T}^2$ ($L_f$), but in general the manifolds $W^u$ do not form an invariant lamination because different elements of this family may be intersect. However,  the manifolds $\hat{W}^u$ do form an invariant lamination of $L_f$. Besides, in \cite{L} and \cite{LS} were proved absolute continuity properties for the laminations $W^s$ and $\hat{W}^u$ respectively. We state the results for the sake of completeness. 
\begin{lemma}\label{conabsstable}
Given any Pesin block $\Lambda_k$, the holonomy of local stable manifolds of points in $\Lambda_k$ between any two transversals is absolutely continuous w.r.t. the Lebesgue measure of the two transversals.     
\end{lemma}
\begin{lemma}\label{conabsunstable}
Given any partition of $L_f$ subordinated to the Pesin unstable lamination $\hat{W}^u$, the disintegrations $\hat{m}_u$ along the elements of the partition are absolutely continuous w.r.t. the Lebesgue measure on the unstable manifolds.
\end{lemma}

\subsection{Dynamical and ergodic properties for $f_t$}

Let consider a linear map $E$ having $\pm1$ as an eigenvalue. In this case, the remainder eigenvalue is $m=\det E$, because the determinant of a linear map is the product of its eigenvalues. In particular, $E$ is diagonalizable and, up to a linear change of coordinates, it is written as 
\begin{displaymath}
E=
\left( \begin{array}{cc}
m & k(m-1)\\
0 & 1
\end{array} \right), \quad m>2,
\end{displaymath}
where $k\in\mathbb{N}$ is chosen in a such way that $E$ no fixes $e_2$. Moreover,  
\begin{displaymath}
    E^{-1}(\overline{x})=\overline{y}+\left\lbrace \left(\frac{i}{m},0\right) 
 : i=0,\ldots m-1\right\rbrace,\quad\forall \overline{x}\in\mathbb{T}^2,
\end{displaymath}
where $y$ is the unique point in $\mathbb{R}^2$ satisfying $Ey=x$. Moreover, by Theorem 1 and Theorem 2 in \cite{R}, it follows that the elementary divisors of $E$ are $\tau_1=1$ and $\tau_2=m$. Since $m\geq 3$, by Theorem \ref{mainteo} there are NUH elements in its homotopy class. Furthermore, those elements have the form $f_t=E\circ h_t$, $t>0$, where $h_t(\overline{x})=(x,y+ts(x))$, for every $\overline{x}=(x,y)\in\mathbb{T}^2$ and $s:\mathbb{T}^1\to\mathbb{R}$ is an analytic function with constant derivatives on $\mathcal{G}$. In addition, if we set  $\delta=2t^{-\frac{3}{10}}$ in the definition of critical region, we have for $t>0$ large enough that  
\begin{equation}\label{estimacriticalzone}
k(m-1)\vert s'(x)\vert<4t^{-\frac{3}{10}},\forall\overline{x}\in\mathcal{C}\quad\text{ and }\quad\vert s'(x)\vert\geq \frac{t^{-\frac{3}{10}}}{2},\quad\forall \overline{x}\in\mathcal{C}\setminus\mathcal{C}_{\delta},
\end{equation}
where $\mathcal{C}_{\delta}=\left(\bigcup_{j=0}^{\tau_2}\left[\frac{j}{\tau_2+1}-\frac{\delta}{2},\frac{j}{\tau_2+1}+\frac{\delta}{2}\right]\right)\times\mathbb{T}^1$. Similar estimations were given in \cite{Oba2018} for the standard map. In coordinates, $Df_t=E\circ Dh_t$ can be written as 
\begin{equation}\label{definohiperbolico}
    Df_t(\overline{x})=
\left( \begin{array}{cc}
m+k(m-1)ts'(x) & k(m-1)\\
ts'(x) & 1
\end{array} \right),\quad\forall \overline{x}=(x,y)\in\mathbb{T}^2. 
\end{equation}
Note that $\det\vert Df_t(\overline{x})\vert=m$, for every $\overline{x}\in\mathbb{T}^2$. In addition, there is $C>1$ such that for any $t>0$ and $\overline{x}\in\mathbb{T}^2$,
\begin{equation}\label{estifaderivadaendo}
    (Ct)^{-1}\leq min(Df_t(\overline{x}))\leq\Vert Df_t(\overline{x})\Vert\leq Ct\quad\text{and}\quad\Vert D^2f_t(\overline{x})\Vert\leq C^2t, 
\end{equation}
where $D^2g$ denotes the Hessian of $g$. 

Let consider $\delta_0\in(0,1)$ and $m_0\in\mathbb{N}$ as in \eqref{estimatet} and \eqref{n}. The next two lemmas establish the existence of unstable manifolds of large size with good estimates for the tangent directions on a positive Lebesgue measure subset of $\mathbb{T}^2$. Their proofs follow the ideas presented in \cite{Oba2018}, which, in turn, build upon the proof of \cite[Theorem 5]{CP18}, requiring an adaptation of the Pliss Lemma. It should be noted that, although the author in \cite{Oba2018} deals with a partially hyperbolic system on $\mathbb{T}^4$, the key tool he utilizes in his argument for the existence of large unstable manifolds is the estimation of bounds for the central Lyapunov exponents, as presented in Proposition 1.4 of the aforementioned reference. In that setup, the central 2-dimensional dynamics involve a Standard map, and the bounds on the exponents are analogous to those given in \cite[Remark 3.6]{ACS}. To adapt these previous ideas to our context, the choice of a large $m_0$ in \eqref{n} and the use of the dynamics $\hat{f}$ on $L_f$ become essential. 

\begin{lemma}\label{Oba18}
Let $\delta_0\in(0,1)$ small enough. Let $E$ be a linear map  whose elementary divisors are $(1,m)$, where $m$ satisfies $\eqref{n}$, and let $f_t=E\circ h_t$ as above.
Then, for $t$ large enough there is a set $Z_t\subset\mathbb{T}^2$ of positive Lebesgue measure such that for every $\overline{x}\in Z_t$ there exists a $C^1$ curve $W_r^+(\overline{x})$ whose length are bounded bellow by $r=t^{-7}$. Moreover, 
\begin{displaymath}
    T_{\overline{y}}W_{r}^+(\overline{x})\subset\Delta^h_{\frac{4}{\theta_1}}\quad\forall\overline{y}\in W_{r}^+, 
\end{displaymath}
where $\theta_1=\theta_1(t)\in(0,1)$. 
\end{lemma}
\begin{proof}
First, by \cite[Remark 3.6]{ACS}, we see that if $m$ satisfies the condition \eqref{n}, one has 
\begin{displaymath}
\min\lbrace\chi_{f_t}^+(\overline{x}),-\chi_{f_t}^-(\overline{x})\rbrace>(1-\delta_0)\log t,\quad \mu-a.e. \overline{x}\in\mathbb{T}^2.  
\end{displaymath}
Thus, by \eqref{relaxp} there is a full $\hat{\mu}$-measure subset $\hat{\mathcal{R}}\subset L_f$ such that the above relation holds on $L_f$, i.e.,
\begin{equation}\label{minexpo}
\min\lbrace\chi_{\hat{f}_t}^+(\hat{x}),-\chi_{\hat{f}_t}^-(\hat{x})\rbrace>(1-\delta_0)\log t,\quad \hat{\mu}-a.e. \hat{x}\in L_f.   
\end{equation}

Now, since $\hat{\mu}$ is hyperbolic  there are at most countable many ergodic components for $\hat{\mu}$. For any ergodic component $\hat{\nu}_i$ for $\hat{\mu}$ define the following set: 
\begin{displaymath}
    \hat{\Lambda}_i=\left\lbrace \hat{x}\in \hat{\mathcal{R}} : \frac{1}{n}\sum_{j=0}^{n-1}\delta_{\hat{f}_t^{j}(\hat{x})},\frac{1}{n}\sum_{j=0}^{n-1}\delta_{\hat{f}_t^{-j}(\hat{x})}\xrightarrow[n\to\infty]{}\hat{\nu}_i,\text{ in the weak* topology}\right\rbrace,
\end{displaymath}
where $\delta_z$ is the dirac measure on the point $z$. It is easy to check that $\hat\Lambda_i\cap\hat\Lambda_j=\emptyset$ if $i\neq j$. Let 
\begin{displaymath}
    \hat{\Lambda}=\bigcup_{i\in\mathbb{N}}\hat{\Lambda}_i. 
\end{displaymath}
Notice that $\hat{\mu}(\hat{\Lambda})=1$. Define the sets
\begin{eqnarray*}
    \hat{Z}_{i,t}^+&=&\left\lbrace \hat{x}\in\hat{\Lambda}_i : \left\Vert D\hat{f}_t^{-n}(\hat{x})\vert_{E_{\hat{x}}^+}\right\Vert<\left(t^{-\frac{4}{5}}\right)^n,\forall n> 0 \right\rbrace,\\
    \hat{Z}_{i,t}&=&
    \hat{f}_t^{-1}(\hat{Z}_{i,t}^+). 
\end{eqnarray*}
In this way, we set 
\begin{displaymath}
    \hat{Z}_t=\bigcup_{i\in\mathbb{N}}\hat{Z}_{i,t}. 
\end{displaymath}

Note that 
$$
1 \leq \Vert D\hat{f_t}^{-1}(\hat{f_t}(\hat{x}))|_{E^+_{\hat{f_t}(\hat{x})}} \Vert \cdot \Vert D\hat{f_t}(\hat{x})|_{E^+_{\hat{x}}} \Vert < t^{\frac{-4}{5}} \Vert D\hat{f_t}(\hat{x})|_{E^+_{\hat{x}}} \Vert,
$$ 
thus we have $\Vert D\hat{f_t}(\hat{x})|_{E^+_{\hat{x}}} \Vert > t^{\frac{4}{5}}$ for all $\hat{x} \in \hat{Z}_t$.

From the definition of $D\hat{f}$ and \eqref{estifaderivadaendo}, we can follow the argument given in pp. 1023-1024 of \cite{Oba2018} step by step to conclude that for sufficiently large $t>0$, 
\begin{displaymath}
    \hat{\mu}(\hat{Z}_t)\geq \frac{1-7\delta_0}{1+7\delta_0}>0.
\end{displaymath}
The proof is omitted here as it is analogous to the one provided by Obata in \cite{Oba2018}.


So, since $\hat{Z}_t\subset \pi_{ext}^{-1}(\pi_{ext}(\hat{Z}_t))$, it follows that $Z_t=\pi_{ext}(\hat{Z}_t)$ satisfies 
\begin{displaymath}
    \mu(Z_t)\geq \frac{1-7\delta_0}{1+7\delta_0}>0. 
\end{displaymath}

Now, from definition of $Z_t$ one has that for every $\overline{x}\in Z_t$ there is a pre-orbit $\hat{x}=(\overline{x},\overline{x}_1,\ldots,\overline{x}_n,\ldots)\in\hat{Z}_t$ and a unit vector $v^+$ such that $\pi_{ext}(\hat{x})=\overline{x}$ and  
\begin{displaymath}
    \left(\frac{1}{Ctm^{\frac{1}{2}}}\right)^{2n}\leq\frac{\left\Vert (Df_t^{n}(\overline{x}_{n-1}))^{-1}v^+\right\Vert^2}{\vert(\det Df_t^{n}(\overline{x}_{n-1}))^{-1}\vert}\leq \left(\frac{m^{\frac{1}{2}}}{t^{\frac{4}{5}}}\right)^{2n},\quad\forall n\geq0.
\end{displaymath} 

Let $\sigma=t^{-\frac{4}{5}}$, $\widetilde{\sigma}=(Ct)^{-1}$, $\rho=\left(m^{\frac{1}{2}}t^{-\frac{4}{5}}\right)^2$ and $\widetilde{\rho}=(Cm^{\frac{1}{2}}t^{\frac{4}{5}})^{-2}$. Note that $\sigma,\widetilde{\sigma},\rho,\widetilde{\rho}\in(0,1)$. Furthermore, 
\begin{displaymath}
    \frac{\sigma\cdot\widetilde{\rho}}{\widetilde{\sigma}\cdot\rho}=\frac{1}{C^3}t^{-\frac{3}{5}}>t^{-\frac{4}{5}}=\sigma,
\end{displaymath}
for $t>0$ large enough. Therefore, since $f_t$ is a local diffeomorphism, by using \eqref{estifaderivadaendo} and the above estimates, we follow step by step the argument given in the proof of Lemma 3.7 and Proposition 3.11 in \cite{Oba2018} to obtain the $C^1$ curve $W_r^+(\overline{x})$, for every $\overline{x}\in Z_t$, satisfying the desired properties with $\theta_1=\theta_1(t)=t^{-\frac{2}{5}}$, for $t>0$ large enough. 
\end{proof} 

Now, let consider the set $Z_t$ given in the above lemma, and let $T=\left\lfloor\frac{1+7\delta_0}{28\delta_0}\right\rfloor$. By \eqref{estimatet} one has $T>20$. In this case, define 
\begin{displaymath}
    X_{\mu}=\bigcap_{j=0}^{T-1}f_t^{-j}(Z_t). 
\end{displaymath}
Since $\mu(Z_t)\geq \frac{1-7\delta}{1+7\delta_0}$, one has $\mu(\mathbb{T}^2\setminus Z_t)\leq \frac{14\delta_0}{1+7\delta_0}$. Hence, 
\begin{eqnarray*}
    \mu(X_{\mu})&=&1-\mu(\mathbb{T}^2\setminus X_{\mu}) \\
    &\geq&1-\sum_{j=0}^{T-1}\mu(f_t^{-j}(\mathbb{T}^2\setminus Z_t))\\
    &\geq&1-T\cdot\frac{14\delta_0}{1+7\delta_0}\\
    &\geq&\frac{1}{2}>0. 
\end{eqnarray*}

On the other hand, let $\theta_1$ as in Lemma \ref{Oba18}, and let $\theta_2=\theta_2(t)=t^{-\frac{3}{5}}$. Note that $\theta_2<\theta_1<1$. On the other hand, by definition of $\delta$ and $\mathcal{C}_{\delta}$, 
\begin{displaymath}
    d(\partial\mathcal{G}, \partial\mathcal{C}_{\frac{\delta}{2}})=t^{-\frac{3}{10}}>t^{-7}=r.  
\end{displaymath}
By the estimations given in \eqref{estimacriticalzone} we have the following properties for the set $Z_t$: 
\begin{enumerate}[(Z1)]
    \item $Z_t\subset\mathcal{G}\subset \mathbb{T}^2\setminus\mathcal{C}_{{\frac{\delta}{2}}}=\mathcal{G}'$. 
    \item If $\overline{x}\in\mathcal{G}'$, then
    \begin{displaymath}
        Df_t(\overline{x})\left(\Delta_{\frac{4}{\theta_1}}^h\right)\subset\Delta_{\theta_2}^h. 
    \end{displaymath}
    \item If $\overline{x}\in\mathcal{G}'$ and $u\in\Delta^h_{\theta_2}$, then $\Vert Df_t(\overline{x})u\Vert\geq t^{\frac{1}{2}}$. In particular,   for any $C^1$-curve $\gamma\subset\mathcal{G}'$ tangent to $\Delta_{\theta_2}^h$ satisfying $\ell_m(\gamma)\geq t^{-\frac{3}{10}}$, one has 
    \begin{displaymath}
        \ell_m(f_t(\gamma))>4. 
    \end{displaymath}
\end{enumerate}

\begin{lemma}\label{last}
For $\overline{x}\in X_{\mu}$, there are $n\in\mathbb{N}$ and a $C^1$ curve $\gamma_+\subset f_t^n(W_r^+(\overline{x}))$ tangent to $\Delta_{\theta_2}^h$ that it crosses   $\mathbb{T}^2$ horizontally. 
\end{lemma}
\begin{proof}
    Since $\overline{x}\in X_{\mu}$, one has by property (Z1),
\begin{displaymath}
    \overline{x}, f_t(\overline{x}),\ldots, f_t^{T-1}(\overline{x})\in\mathcal{G}\subset \mathcal{G}'. 
\end{displaymath}
In this way, denote $W_0^+=W_r^+(\overline{x})$ and $W^+_k(\overline{x})=f_t^k(W_0^+)$ for $k\geq1$. Note that since $\overline{x}\in Z_t$ and $r<d(\partial\mathcal{G},\mathcal{G}')$ we have that $W_0^+\subset\mathcal{G}'$. Then, by property (Z2), $TW_1^+\subset\Delta_{\theta_2}^h$. Hence, since $\ell_m(W_1^+)>\ell_m(W_0^+)$ and $f_t(\overline{x})\in\mathcal{G}$, it follows that there is a connected component $\gamma_1$ of $W_1^+$ contained in $\mathcal{G}$ such that  $\overline{x}_1=f_t(\overline{x})\in\gamma_1$, it is tangent to $\Delta_{\theta_2}^h$ and its length is $t^{-7}$.

Now, let  
\begin{displaymath}
    p_0=\max\lbrace p\in\mathbb{N} : f_t^j(\gamma_1)\subset\mathcal{G},\;\forall j=1,\ldots, p\rbrace.
\end{displaymath}
By property (Z3), the definition of $\gamma_1$ and $X_{\mu}$ and the invariance of the cone $\Delta_{\theta_2}^h$ on $\mathcal{G}'$ we see that the curve $\gamma_j=f_t^j(\gamma_1)$, $j\geq 1$, satisfies $\ell_m(\gamma_j)\geq t^{\frac{1}{2}}\ell_m(\gamma_1)=t^{\frac{j-14}{2}}$, so that $p_0\leq 14$. 

    Let $p_0^+=p_0+1\leq 15<T-1$. Then, $f_t^{p_0^+}(\overline{x})\in\mathcal{G}\subset\mathcal{G}'$ and $\gamma_{p_0^+}\cap\partial\mathcal{G}'\neq\emptyset$. In particular, $\gamma_{p_0^+}$ also intersects $\partial\mathcal{G}$. Thus, there is a connected component $\hat{\gamma}_{p_0^+}$ of $\gamma_{p_0^+}$ contained in $\mathcal{G}'\setminus\mathcal{G}$ such that $\hat{\gamma}_{p_0^+}\cap\mathcal{G}\neq\emptyset$ and $\hat{\gamma}_{p_0^+}\cap\mathcal{G}'\neq\emptyset$, which implies that its length is at least $t^{-3/10}$, so that, by property (Z3), the $C^1$ curve $\gamma_+=f_t(\hat{\gamma}_{p_0^+})$ satisfies  $\ell_m(\gamma_+)>4$. Moreover, $\gamma_+\subset W_{p_0+3}^+(\overline{x})$ and $T\gamma_+\subset\Delta_{\theta_2}^h$. 

    Finally, if $\gamma_+(t)=(\gamma_1(t),\gamma_2(t))$, $t\in I$, we have that $\vert\gamma_1'(t)\vert\geq\theta_2\vert\gamma_1'(t)\vert\geq\vert\gamma_2'(t)\vert$ for every $t\in I$, so that $\Vert \gamma'_+\Vert=\vert\gamma_1'(t)\vert$. Therefore, the projection of $\gamma^+$ to the horizontal axis, denoted by $\pi_{h}(\gamma^+)$, satisfies 
    \begin{displaymath}
        \ell(\pi_{h}(\gamma^+))=\int_{I}\vert\gamma_1'(t)\vert dt=\int_{I}\Vert\gamma_+'(t)\Vert dt=\ell_m(\gamma^+)>1,
    \end{displaymath}
    which proves the result. 
\end{proof}

Finally, let $\hat{\Lambda}_0\subset\hat{\mathcal{R}}$ be as defined in the proof of Lemma \ref{Oba18}. As in \cite{ACS}, let $\hat{\Lambda}_1\subset\mathcal{R}$ be the set of points such that for every $\hat{x}\in\hat{\Lambda}_1$, there is a full Lebesgue measure subset $B$ of $\hat{W}^u(\hat{x})$ such that $B\subset\hat{\Lambda}_0$. By Lemma \ref{conabsunstable} we have that $\hat{\mu}(\hat{\Lambda}_1)=1$. Moreover, by definition of $\hat{\Lambda}_0$, there is an ergodic component $\hat{\mu}_0$ of $\hat{\mu}$ such that $B\subset\mathcal{B}(\hat{\mu}_0)$. In this case, we define $\Lambda_j=\pi_{ext}(\hat{\Lambda}_j)$, $j=0,1$. Notice that these sets are forward invariant and have full Lebesgue measure. Furthermore, if $\overline{x}\in\Lambda_1$, there are $\hat{x}\in\hat{\Lambda}_1$ and an ergodic component $\hat{\mu}_0$ of $\hat{\mu}$ such that $\pi_{ext}(\hat{x})=\overline{x}$ and Lebesgue almost every point in $W^u(\hat{x})$ belongs to $\mathcal{B}(\mu_0)$, where $\mu_0=(\pi_{ext})_*\hat{\mu}_0$.   

In \cite{ACS} the authors introduce the notion of $\mu_0$-regular $su$-rectangle for a endomorphism of class $C^2$ as a piecewise smooth simple closed curve in $\mathbb{T}^2$, consisting of two pieces of local stable manifolds and two pieces of local unstable manifolds such that the last two pieces are contained in $\Lambda_1$, and almost every point in these pieces belongs to the basin of $\mu_0$. In the aforementioned reference, they shown that when the Lebesgue measure is hyperbolic and almost every stable manifold has large diameter, these rectangles are open modulo zero subsets of $\mathbb{T}^2$. 

\subsection{Stable ergodicity} 

Next, we are ready to prove Theorem B. 

\begin{proof}[proof of Theorem \ref{main2}]
First, notice that in the proof of non-uniform hyperbolicity presented in \cite{ACS} the following estimate for $C_{\chi}(f_t)$ is obtained in our setting:  
\begin{displaymath}
    C_{\chi}(f_t)\geq \left(\frac{\lfloor\frac{m-1}{2}\rfloor-1}{\lfloor\frac{m-1}{2}\rfloor+1}\right)\log t+C_0,\quad C_0\in\mathbb{R}. 
\end{displaymath}
Let $\varepsilon>0$ small enough, and let $t>0$ and $m\geq m_0$ satisfying 
\begin{equation}\label{m0}
\left(\frac{\lfloor\frac{m-1}{2}\rfloor-1}{\lfloor\frac{m-1}{2}\rfloor+1}\right)\log t+C_0-\varepsilon>\left(\frac{\lfloor\frac{m_0-1}{2}\rfloor-1}{\lfloor\frac{m_0-1}{2}\rfloor+1}\right)\log t.   
\end{equation}
In this case, take $f_t=E\circ h_t\in\mathcal{U}$, where $E$ has elementary divisors $(1,m)$. Consider the following properties satisfied by $Dh_t$ and $f_t$: 
\begin{enumerate}[(1)]
    \item For any $\overline{x}\in\overline{\mathcal{G}}$ we have 
    \begin{itemize}
    \item $(Dh_t(\overline{x}))^{-1}(\Delta_{\alpha}^h)\subset\overline{(Dh_t(\overline{x}))^{-1}(\Delta_{\alpha}^h)}\subset\Delta_{\alpha}^v$,
    \item $Dh_t(\overline{x})(\Delta_{\alpha}^v)\subset\overline{Dh_t(\overline{x})(\Delta_{\alpha}^v)}\subset\Delta_{\alpha}^h$,
    \item $m((Dh_t(\overline{x}))^{-1}),m(Dh_t(\overline{x}))>\frac{at-\alpha}{\alpha}$.
    \end{itemize}
    \item For any $(\overline{x},u)\in T^1\mathbb{T}^2$, there are $\overline{y}\in\mathcal{G}$ and $w\in\Delta_{\alpha}^v$ such that $f_t(\overline{y})=\overline{x}$ and $Df_t(\overline{y})w=u$.
    \item For any $\overline{x}\in\mathbb{T}^2$, there is a pre-image $\overline{y}\in\mathcal{G}$ with $d(\overline{y},\mathcal{C})>\frac{1}{8(m+1)}$. 
\end{enumerate}
In this way, define 
\begin{displaymath}
    \mathcal{V}'=\lbrace g=E\circ h : h\in \text{Diff}^2(\mathbb{T}^2)\text{ is }C^2\text{-close to }h_t\text{ and satisfies (1)}\rbrace, 
\end{displaymath}
and we choose $\mathcal{V}\subset\mathcal{V}'$ such that every $g\in\mathcal{V}$ satisfies properties (2) and (3) above. Observe that $\mathcal{V}$ is a $C^2$ open subset of $End^2(\mathbb{T}^2)$ containing $f_t$. Take $g\in \mathcal{W}:=\mathcal{U}\cap\mathcal{V}$. Then, by the choice of $\alpha>1$, the definition of $\mathcal{W}$ and \cite[Lemma 6.2]{ACS}, one has that $g$ is an area-preserving NUH endomorphism such that for any $\overline{x}\in\mathbb{T}^2$ and every $C^1$ curve $\gamma_x$ passing through $\overline{x}$, there is $N'\in\mathbb{N}$, a pre-image $\overline{y}\in\mathbb{T}^2$ and a $C^1$ curve $\gamma'_x$ passing through $\overline{y}$ which crosses $\mathbb{T}^2$ vertically such that $g^{N'}(\gamma'_x)=\gamma_x$. In particular, from \cite[Lemma 6.3]{ACS}, it follows that $W_g^s(\overline{x})$ contains a $v$-segment for almost every point $\overline{x}\in\mathbb{T}^2$.

On the other hand, by shrinking $\mathcal{W}$ if it is necessary, similar estimations to \eqref{estifaderivadaendo} and properties (Z1)-(Z3) are obtained for any $g\in\mathcal{W}$. Moreover, we have by definition of $C_{\chi}(f_t)$, $m_0$ and relation \eqref{m0} that 
\begin{displaymath}
 C_{\chi}(g)\geq \left(\frac{\lfloor\frac{m-1}{2}\rfloor-1}{\lfloor\frac{m-1}{2}\rfloor+1}\right)\log t+C-\varepsilon>\left(\frac{\lfloor\frac{m_0-1}{2}\rfloor-1}{\lfloor\frac{m_0-1}{2}\rfloor+1}\right)\log t>(1-\delta_0)\log t,
\end{displaymath}
so that, by \cite[Proposition 2.2]{ACS}, $\min\lbrace\chi_{g}^+(\overline{x}),-\chi_{g}^-(\overline{x})\rbrace>(1-\delta_0)\log t$ for Lebesgue almost every point $\overline{x}\in\mathbb{T}^2$. Therefore, Lemma \ref{conabsstable}, Lemma \ref{conabsunstable}, Lemma \ref{Oba18} and Lemma \ref{last} hold for every $g\in\mathcal{W}$. 

Define the sets $X_g$ and  $\Lambda_1^g$ in an analogous way to that given for $f_t$.  Assume that $\mu$ has at least two different ergodic components $\mu_0$ and $\mu_1$. On one hand, since $\Lambda_1^g$ has full Lebesgue measure and $X_g$ has positive Lebesgue measure, there is $\overline{x}\in X_g\cap\Lambda_1^g$. Suppose that $\overline{x}\in\mathcal{B}(\mu_0)$. Then, by Lemma \ref{Oba18} and Lemma \ref{last} there are a $C^1$-curve $\gamma_x$ containing $\overline{x}$ and a natural number $n_0<20$ such that $\gamma^{n_0}=g^{n_0}(\gamma_x)$ crosses $\mathbb{T}^2$ horizontally. Moreover, by the topological characterization of the Pesin unstable mannifold, we have $W_r^+(\overline{x})\subset W^u(\hat{x})$, where $\hat{x}\in\Lambda_1^g$ satisfies $\pi_{ext}(\hat{x})=\overline{x}$, so that there is a full Lebesgue measure subset $B$ of $\mathcal{B}(\mu_0)$ contained in $\gamma^{n_0}$.
On the other hand, let $R$ be a $\mu_1$-regular $su$-rectangle, and take $\overline{y}\in int(R)\cap\mathcal{B}(\mu_1)$. Then, there is a small $C^1$-curve $\gamma_y\subset R$ containing $\overline{y}$. Hence, there exists $N\in\mathbb{N}$ and a $C^1$-curve $\gamma_y'$ crossing $\mathbb{T}^2$ vertically such that $g^{N}(\gamma_y')=\gamma_y$. The Figure 1 helps visualize the argument up to this point.  
\begin{figure}[ht]
\includegraphics[scale=0.3]{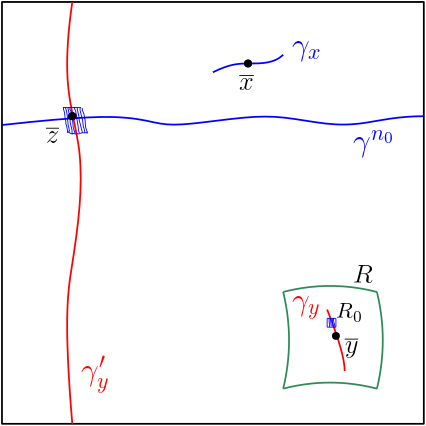}
\caption{Proof of Theorem B.}
\end{figure}

Therefore, as $\gamma^{n_0}$ and $\gamma_y'$ cross $\mathbb{T}^2$ horizontally and vertically,  respectively, there is a point $\overline{z}\in\gamma^{n_0}\cap\gamma_y'$. Hence, since $g$ is a local diffeomorphism, we can conclude by the absolute continuity of the stable lamination (Lemma \ref{conabsstable}) that there is a positive Lebesgue measure subset $R_0\subset B(\mu_0)\cap R$, which is impossible. Thus, $\mu$ is ergodic for every $g\in\mathcal{W}$, which proves the result. 
\end{proof}

\subsection*{Acknowledgments}
We would like to thank R. Saghin and the anonymous referees for their valuable comments and suggestions.

\end{document}